\documentclass[12pt,reqno]{amsart}
\usepackage{amsmath,amssymb,enumerate}
\usepackage[dvipdfmx]{graphicx}
\usepackage{color}
\usepackage{comment}
\usepackage[all]{xy}
\usepackage{bm}
\usepackage{dsfont}
\usepackage{autobreak}

\newtheorem{tm}{Theorem}[section]
\newtheorem{lm}[tm]{Lemma}

\newtheorem{re}[tm]{Remark}
\newtheorem{df}[tm]{Definition}
\newtheorem{pr}[tm]{Proposition}

\newtheorem{ex}[tm]{Example}

\makeatletter
\newcommand{\subscripts}[3]{%
  \@mathmeasure\z@\displaystyle{#2}%
  \global\setbox\@ne\vbox to\ht\z@{}\dp\@ne\dp\z@
  \setbox\tw@\box\@ne
  \@mathmeasure4\displaystyle{\copy\tw@_{#1}}%
  \@mathmeasure6\displaystyle{{#2}_{#3}}%
  \dimen@-\wd6 \advance\dimen@\wd4 \advance\dimen@\wd\z@
  \hbox to\dimen@{}\mathop{\kern-\dimen@\box4\box6}%
}
\makeatother

\newcommand{\A}{\mathcal{A}}

\newcommand{\R}{\mathbb{R}}
\newcommand{\N}{\mathbb{N}}

\newcommand{\dd}{\mathrm{d}}
\newcommand{\e}{\mathrm{e}}
\newcommand{\x}{\bm{\mathrm{x}}}
\newcommand{\y}{\bm{\mathrm{y}}}
\newcommand{\ve}{\varepsilon}
\newcommand{\dis}{\displaystyle}
\newcommand{\del}{\partial}
\newcommand{\ol}{\overline}

\newcommand{\nn}{\nonumber}

\newcommand{\E}{\mathbb{E}}

\newcommand{\la}{\langle}
\newcommand{\ra}{\rangle}
\newcommand{\bx}{\mathbf{x}}
\newcommand{\by}{\mathbf{y}}
\newcommand{\bk}{\mathbf{k}}

\newcommand{\bi}{\mathbf{i}}
\newcommand{\bj}{\mathbf{j}}

\newcommand{\Dom}{\mathrm{Dom}}

\newcommand{\PP}{\mathbf{P}}

\allowdisplaybreaks[3]

\makeatletter
 
 \@addtoreset{equation}{section}
\makeatother

\begin{document}
\title[Iterates of multidimensional Bernstein-type operators]
{
{ Iterates of multidimensional Bernstein-type operators
and diffusion processes in population genetics}
}
\author[T. Hirano]{Takatoshi Hirano}
\author[R. Namba]{Ryuya Namba}

\address[T. Hirano]{Department of Risk Management Technology, Mizuho-DL Financial Technology Co., Ltd., 2-4-1, 
Koji-machi, Chiyoda-ku, 
Tokyo, 102-0083, Japan}
\email{{\tt{takatoshi-hirano@fintec.co.jp}}}

\address[R. Namba]{Department of Mathematics,
Faculty of Science,
Kyoto Sangyo University, Motoyama, Kamigamo, Kita-ku, Kyoto, 603-8555, Japan}
\email{{\tt{rnamba@cc.kyoto-su.ac.jp}}}


\maketitle 

\begin{abstract}
The Bernstein operator is known as a
typical example of positive linear operators which 
uniformly approximates continuous functions on $[0,1]$. 
In the present paper, we introduce a multidimensional 
extension of the Bernstein operator which is associated with
a transition probability of a certain discrete Markov chain.  
In particular, 
we show that the iterate of the multidimensional Bernstein-type  
operator 
uniformly converges to the Feller semigroup corresponding to 
the multidimensional Wright--Fisher diffusion process with mutation
arising in the study of population genetics, together with 
its rate of convergence. 
The convergence of process-level is obtained as well. 
Moreover, by taking the limit as both the number of iterate
and the dimension of the Bernstein-type operator 
tend to infinity simultaneously, 
we prove that the iterate of the multidimensional Bernstein-type 
operator uniformly converges to the Feller semigroup 
corresponding to a probability measure-valued Fleming--Viot process
with mutation. 

\end{abstract}


\section{{\bf Introduction}}
\label{Sect:Intro}

\subsection{Background and motivation}

    The study of positive linear operators, 
    acting on some function spaces, is
    one of central themes in approximation theory, 
    functional analysis, probability theory
    and so on. 
    In particular, the property that such operators uniformly 
    approximate continuous functions
    has been investigated extensively. 
    Let $C([0, 1])$ be the space of continuous functions
    $f : [0, 1] \to \R$. 
    One of well-known positive linear operators possessing 
    the approximating property 
    is the Bernstein operator, 
    which is defined as follows:
    
    \begin{df}[cf.~\cite{Bernstein}]
    \label{Def:operators-1D}
       For $n \in \N$, a positive linear operator 
       $B_n$ acting on $C([0, 1])$ is defined by 
       \begin{equation}\label{Eq:Bernstein}
       B_n f(x) 
        = \sum^{n}_{k=0}\binom{n}{k}x^k (1-x)^{n-k} f\left(\frac{k}{n}\right),
        \qquad f \in C([0,1]), \, x \in [0,1].
       \end{equation}
       The operator $B_n$ is called the Bernstein operator. 
    \end{df}
    
    Indeed, the Bernstein  operator $B_n$ has the following approximating property.

    \begin{pr}[cf.~\cite{Bernstein}]
    \label{Prop:approximation}
    For any $f \in C([0, 1])$, we have
    \[
    \lim_{n \to \infty}\max_{x \in [0, 1]}|B_nf(x)-f(x)|=0. 
    \]
    \end{pr}
    \noindent
    It is well known that Proposition \ref{Prop:approximation}  provides 
    a constructive proof of 
    the celebrated {\it Weierstrass approximation theorem}. 
    In fact, we can show it as an application of 
    the weak law of large numbers for binomial random variables. 
    See e.g.,  \cite[Example 5.15]{Klenke}.

    Let us now consider the $k$ times iteration $B_n^k$ of the Bernstein operator 
    $B_n$ itself. Kelisky and Rivlin  first focused 
    on the limiting behavior 
    of $B_n^k$ as $k \to \infty$ with fixed $n$ 
    and obtained in \cite[Theorem 1]{KR67} that 
    $B_n^k f, \, f \in C([0, 1]),$
    uniformly converges to the linear function 
    which interpolates $f(0)$ and $f(1)$. 
    After their result, several limiting behaviors 
    of $B_n^k f$ as both $k$ and $n$
    tend to infinity under some relations 
    between $k$ and $n$ have been discussed. 
    For instance, Karlin and Ziegler investigated 
    long time behaviors of the $k(n)$-times iterates of positive linear 
    operator as $n \to \infty $ and the limiting semigroup is explicitly identified
    for various classes of positive linear operators in \cite{KZ}.
    Konstantopoulos, Yuan and Zazanis
    showed in \cite[Theorem 3]{KYZ18} that 
    $B_n^{\lfloor nt \rfloor}f$, $f \in C([0, 1])$,
    uniformly converges to the diffusion process which solves 
    the stochastic differential equation
    \begin{equation}\label{Eq:onedim-Wright-Fisher}
            \dd {\sf X}_t(x)=\sqrt{{\sf X}_t(x)(1-{\sf X}_t(x))} \, \dd W_t, 
            \qquad {\sf X}_0(x)=x \in [0, 1],
        \end{equation}
        where $ (W_t)_{t \ge 0} $ is a one-dimensional 
        standard Brownian motion. 
    The diffusion process $ \big({\sf X}_t(x)\big)_{t \ge 0} $ is called 
    the {\it Wright--Fisher diffusion}, 
    which appears as a fundamental model for the evolution of the 
    allele frequency in population genetics. See e.g., 
    \cite{EK} for more details. 
    In the proof of \cite[Theorem 3]{KYZ18}, 
    the authors applied a fully stochastic approach based on 
    It\^o calculus, which provides a probabilistic interpretation
    of the limiting behavior of the iterates of $B_n$. 
    
    One of remarkable properties of the Bernstein operator $B_n$ is  
    its probabilistic expression in terms of the expectation of a 
    binomial random variable. 
    Indeed, suppose that $X_i(x), \, i=1, 2, 3, \dots,$ are independent and
    identically distributed Bernoulli random variables 
    with $\mathbb{P}(X_1(x)=1)=x$ and $\mathbb{P}(X_1(x)=0)=1-x$ 
    for $x \in [0, 1]$.
    Then, the random variable  
    $S_n(x):=X_1(x)+X_2(x)+\cdots+X_n(x)$ is binomial 
    and we can see that 
    \begin{equation}\label{Bernstein-expectation}
    B_nf(x)=\mathbb{E}\left[ f\left(\frac{1}{n}S_n(x)\right) \right], 
    \qquad \, f \in C([0, 1]), \, x \in [0, 1]. 
    \end{equation}
    This expression highly plays a crucial role 
    in obtaining such a limiting behavior
    from a probabilistic perspective. 
    Therefore, it is quite natural to 
    consider the multidimensional versions 
    of the Bernstein operator by replacing the binomial random variable $S_n(x)$ with the 
    multinomial one,
    which motivates some further investigations of interest.

\subsection{Aim and overview of main results}

    Stimulated by the circumstance mentioned above, 
    we aim to find out multidimensional versions of 
    limit theorems for iterates of the
    multidimensional Bernstein operator 
    and to reveal various stochastic phenomena behind the iterates.
    By the aid of the multinomial distribution, 
    we can extend \eqref{Eq:Bernstein} 
    to the multidimensional cases. 

    We put $\N_0:=\N \cup\{0\}$. 
    Let $d \ge 2$ be an integer. 
    We define 
    the $(d-1)$-simplex in $\R^d$ by
    \[
    \Delta_{d-1}:=\left\{
        \bx = (x_1, x_2, \dots, x_d) \in \R^d \,
        | \,  x_i \ge 0, i=1, 2, \dots, d, \, |\bx| = 1 
        \right\}.
    \]
    By the symbol $C(\Delta_{d-1})$, we mean 
    the Banach space of all continuous functions $f : \Delta_{d-1} \to \R$,
    which we endow with the usual supremum norm 
    \[
    \| f\|_\infty:=\max_{\bx \in \Delta_{d-1}}|f(\bx)|,
    \qquad f \in C(\Delta_{d-1}).
    \]
    
    \begin{df}[cf.~\cite{Dinghas}]
    \label{Def:operators-multiD}
       For $n \in \N$, $f \in C(\Delta_{d-1})$ and $\bx \in \Delta_{d}$,
       the positive linear operator 
       $B_{d,n}$ acting on $C(\Delta_{d-1})$ is defined by 
       \begin{equation}\label{Eq:d-Bernstein}
       B_{d, n} f(\bx) 
        = \sum_{\bk \in \N_0^d, |\bk| = n} 
            \frac{n!}{k_1!k_2!\cdots k_d!}
            x_1^{k_1}x_2^{k_2}\cdots x_d^{k_d}
            f \left(\frac{\bk}{n}\right).
       \end{equation}
       The operator $B_{d, n}$ is called the $d$-dimensional Bernstein operator. 
    \end{df}

To our best knowledge, 
the operator $B_{d, n}$ was first introduced in 
\cite{Dinghas}.  
The approximating property of this
operator is also given as in the following.

\begin{pr}[cf.~{\cite[Thoerem 4.9]{Altomare}}]
    \label{prop:approximation-d}
    For any $f \in C(\Delta_{d-1})$, we have
    \[
    \lim_{n \to \infty}\|B_{d,n}f-f\|_\infty=0. 
    \]
    
\end{pr}

Our strategy to obtain limit theorems of interest basically consists of 
two parts. One is to express the iterates of the $d$-dimensional
Bernstein operator in terms of the expectation of 
a multinomial distribution. 
As the matter of fact, 
the operator $B_{d,n}$ has the expression 
$$
B_{d, n}f(\bx)=\mathbb{E}\left[f\left(\frac{1}{n}S_n(\bx)\right)\right], \qquad 
f \in C(\Delta_{d-1}), \, \bx \in \Delta_{d-1},
$$
similarly to \eqref{Bernstein-expectation}, where 
$S_n(\bx)$ denotes the multinomial random variable with 
parameters $n$ and $\bx$.
This expression allows us to relate the iterate $B_{d, n}^k$
to a time-homogeneous discrete Markov chain $\{H_{n}^k(\bx)\}_{k=0}^\infty$
with values in $(1/n)\mathbb{Z} \cap \Delta_{d-1}$. 
Hence, every limit theorem for $B_{d, n}^k$ can read probabilistic that for 
the Markov chain $\{H_{n}^k(\bx)\}_{k=0}^\infty$. 

The other is to employ some powerful tools from functional analysis
such as {\it Trotter's approximation theorem} (cf.~\cite{Trotter, Kurtz}),
which provides a sufficient condition for the iterate of 
a bounded linear operator to converge to a $C_0$-contraction semigroup
on some Banach space. Since we focus on iterates of the multidimensional 
Bernstein operator, the theorem is quite useful in order to deduce desired limit theorems 
for them. Moreover, by applying a refinement of Trotter's approximation theorem
(cf.~\cite{Namba}), 
we can also establish quantitative estimates for the limit theorems simultaneously.
See Propositions \ref{Prop:Trotter} and \ref{Prop:Trotter-rate} for their precise statements.

By following the strategy above, we first show in Theorem \ref{Thm:Bernstein-1} that 
$(B_{d, n})^{\lfloor nt \rfloor}f$, $f \in C(\Delta_{d-1})$, 
uniformly converges to the $d$-dimensional diffusion semigroup corresponding to
\begin{equation}
        \label{SDE:d-Bernstein}
            \dd \mathsf{X}_t^i(\bx)=\sum_{j=1}^d
            \sigma_{ij}\big(\mathsf{X}_t(\bx)\big) \, \dd W_t^j, 
            \qquad \mathsf{X}_0(\bx)=\bx \in \Delta_{d-1}, \qquad 
            i=1, 2, \dots, d,
        \end{equation}
        where $\sigma=(\sigma_{ij})_{i, j=1}^d : \Delta_{d-1} \to \R^d \otimes \R^d$ 
        satisfies that 
        \begin{equation}\label{Eq:diffusion-matrix}
        \sigma(\bx)\sigma(\bx)^\top=\Big(x_i(\delta_{ij}-x_j)\Big)_{i, j=1}^d, 
        \qquad \bx \in \Delta_{d-1},
        \end{equation}
        and 
        $(W_t)_{t \ge 0}=(W_t^1, W_t^2, \dots, W_t^d)_{t \ge 0} $ 
        is a $d$-dimensional standard Brownian motion. 
The process $(\mathsf{X}_t)_{t \ge 0}
=(\mathsf{X}_t^1, \mathsf{X}_t^2, \dots, \mathsf{X}_t^d)_{t \ge 0}$ is called the 
{\it $d$-dimensional Wright--Fisher diffusion}, which is a
$d$-dimensional extension of \eqref{Eq:onedim-Wright-Fisher}.
Here, the dimension $d$ corresponds to the number of possible allele types. 
Moreover, we also establish a strong result that 
the linearly interpolated Markov chain corresponding to 
the iterate of $B_{d,n}$ converges in law to the $d$-dimensional 
Wright--Fisher diffusion in the space of $(1/2-)$-H\"older 
continuous functions on $\Delta_{d-1}$. See Theorem \ref{Thm:Bernstein-2}
for its statement.

On the other hand, in considering the Wright--Fisher diffusions
in the context of population genetics, 
it might be natural to take an effect of {\it mutation} 
among $d$ alleles into account. 
The {\it $d$-dimensional Wright--Fisher diffusion with mutation} is 
given by the solution to the stochastic differential equation
\begin{equation}
        \label{SDE:d-Bernstein-q}
            \dd \mathsf{X}_t^i(\bx)=\sum_{j=1}^d
            \sigma_{ij}\big(\mathsf{X}_t(\bx)\big) \, \dd W_t^j
            +\sum_{j=1}^d q_{ji}\mathsf{X}_t^j\, \dd t, 
            \qquad \mathsf{X}_0(\bx)=\bx \in \Delta_{d-1},
        \end{equation}
for $i=1, 2, \dots, d$, where $q_{ij}$, $i \neq j$, represents 
the intensity of a mutation 
from the allele of type $i$ to that of type $j$, and 
$q_{ii}=-\sum_{j \neq i}q_{ij}$ for $i=1, 2, \dots, d$. 

In order to handle with more complex stochastic phenomena arising population genetics 
through the limit of the iterates of positive linear operators, 
we extend the definition of the multidimensional Bernstein operator in the following.

\begin{df}\label{Def:modified Bernstein}
    Let $n \in \mathbb{N}$ and $\mathbf{q}_n=\{q_{ij}^{(n)}\}_{i, j=1}^d$ be real numbers satisfying
    \begin{equation}\label{Eq:cond-q}
    q_{ij}^{(n)}>0, \quad i \neq j, \quad \text{and} \quad 
    q_{ii}^{(n)}=-\sum_{j \neq i}q_{ij}^{(n)}, 
    \quad i=1, 2, \dots, d. 
    \end{equation}
    For $n \in \N$, $f \in C(\Delta_{d-1})$ and $\bx \in \Delta_{d-1}$,
       a positive linear operator 
       $B_{d,n}^{(\mathbf{q}_n)}$ acting on $C(\Delta_{d-1})$ is defined by 
       \begin{equation}\label{Eq:q-Bernstein}
       B_{d, n}^{(\mathbf{q}_n)} f(\bx) 
        = \sum_{\bk \in \N_0^d, |\bk| = n} 
            \frac{n!}{k_1!k_2!\cdots k_d!}
            (x_1^{(\mathbf{q}_n)})^{k_1}(x_2^{(\mathbf{q}_n)})^{k_2}\cdots 
            (x_d^{(\mathbf{q}_n)})^{k_d}
            f \left(\frac{\bk}{n}\right),
       \end{equation}
       where $\bx^{(\mathbf{q}_n)}
       =(x_1^{(\mathbf{q}_n)}, x_2^{(\mathbf{q}_n)},
       \dots, x_d^{(\mathbf{q}_n)}) \in \Delta_{d-1}$ is given by 
       \[
        x_i^{(\mathbf{q}_n)}:=x_i+\sum_{j=1}^d q_{ji}^{(n)}x_j, \qquad 
        i=1, 2, \dots, d. 
       \]
       The operator $B_{d, n}^{(\mathbf{q}_n)}$ is called 
       the $d$-dimensional Bernstein operator associated with $\mathbf{q}_n$. 
\end{df}

In formally putting $q_{ij}^{(n)} \equiv 0$, 
the operator $B_{d, n}^{(\mathbf{q}_n)}$ is nothing but 
the classical Bernstein operator $B_{d, n}$. 
It is important that, under a quite natural assumption {\bf (Q1)} 
(see Section \ref{Sect:Settings}), the operator $B_{d, n}^{(\mathbf{q}_n)}$ also 
approximates every element in $C(\Delta_{d-1})$ uniformly. 
Moreover, thanks to Assumption {\bf (Q1)}, 
we can prove in Theorem \ref{Thm:Bernstein-general-1} that the iterate  
$(B_{d, n}^{(\mathbf{q}_n)})^{\lfloor nt \rfloor}f$, $f \in C(\Delta_{d-1})$, 
uniformly converges to the diffusion semigroup corresponding to 
\eqref{SDE:d-Bernstein-q}, together with its rate of convergence. 
To our best knowledge, this is the first result in that the high-dimensional 
Wright--Fisher diffusion with mutation is captured as the limit of the iterates
of a positive linear operator, which provides a new insight in the study of 
population genetics, needless to say probability theory, 
approximation theory and functional analysis. 
Furthermore, in Theorem \ref{Thm:Bernstein-general-2}, 
the convergence of process-level is also obtained similarly to the case 
in the absence of mutation effect. 

On the other hand, 
the {\it Fleming--Viot process} is well known 
as one of most famous and important 
examples of measure-valued diffusion processes, 
which has become an active branch 
of probability theory. 
This process was introduced by 
Fleming and Viot themselves in \cite{FV}.
When the number of allele types is uncountably infinite, 
the simplex $\Delta_{d-1}$ is replaced by the set of 
Borel probability measures $\mathcal{P}(E)$ 
on a compact metric space $E$, 
and the infinitesimal generator $\mathfrak{A}$ of 
the Fleming--Viot process, acting on the space $C(\mathcal{P}(E))$ 
of continuous functions on 
$\mathcal{P}(E)$, is given by
\begin{align}
\mathfrak{A}\varphi(\mu)
&=\frac{1}{2}\int_E\int_E\mu(\dd z)\big(\delta_z(\dd w) - \mu(\dd w)\big)
\frac{\delta^2 \varphi(\mu)}{\delta\mu(z)\delta\mu(w)} \nn\\
&\hspace{1cm}+\int_E \mu(\dd z)\mathcal{Q}\left(\frac{\delta \varphi(\mu)}{\delta\mu(\cdot)}\right)(z),
\label{Eq:generator-FV}
\end{align}
where $\delta_z$ denotes the delta measure centered at $z \in E$
and $(\mathcal{Q}, \Dom(\mathcal{Q}))$ is the infinitesimal generator of a Feller semigroup
on $C(E)$, referring to as the {\it mutation operator}. Moreover, we mean that
\[
\begin{aligned}
\frac{\delta \varphi(\mu)}{\delta\mu(z)}&=\lim_{\ve \searrow 0}\frac{1}{\ve}
\big\{\varphi(\mu+\ve\delta_z) - \varphi(\mu)\big\}. \\
\end{aligned}
\]
 
As an infinite-dimensional extension of 
Theorems \ref{Thm:Bernstein-1} and \ref{Thm:Bernstein-general-1}, 
we aim to capture the Fleming--Viot process 
in terms of the limit of the
iterate of the multidimensional Bernstein operator 
by an effective use of Trotter's approximation theorem
under some technical but natural assumptions.
For this purpose, we need to take a {\it discretization} of $E$, that is, 
a sequence $E^{(d)}=\{z_1^{(d)}, z_2^{(d)}, \dots, z_d^{(d)}\} \subset E$, 
$d \ge 2$, 
with $z_i^{(d)} \neq z_j^{(d)}$ whenever $i \neq j$. 
Thanks to this discretization, we come
to discuss a finite-dimensional approximation of the Fleming–Viot process 
in terms of the Bernstein operator on $\Delta_{d-1}$.

If we wish to capture \eqref{Eq:generator-FV} in our framework, 
one might try to introduce 
the ``infinite-dimensional Bernstein operator'' $B_{\infty, n}^{(\mathbf{q}_n)}$ 
in a certain sense. 
However, this direction seems to be not too much better 
in that the multinomial distribution 
with infinite categories may appear in the definition, 
which is to be less treatable. 
Instead, our idea is to replace the dimension $d$ of the Bernstein operator 
by a certain sequence $\{d_n\}_{n=1}^\infty \subset \N$
with $d_n=o(n^{1/8})$ and then
take the limit as $n \to \infty$, 
where the choice of the order of $d_n$ is due to some technical reasons.
Moreover, we need to modify Assumption {\bf (Q1)}
according to the choice of $\{d_n\}_{n=1}^\infty$.
Under the new assumption, say {\bf (Q2)}, together with 
another assumption {\bf (Q3)} for the convergence of 
$\{q_{ij}^{(n)}\}_{i, j=1}^{d_n}$ to the mutation operator $\mathcal{Q}$ 
as $n \to \infty$, 
we establish a limit theorem for the iterate
$(B_{d_n, n}^{(\mathbf{q}_n)})^{\lfloor nt \rfloor}$ 
in Theorem \ref{Thm:Bernstein-infinity-1}. 
The limiting object we capture here 
is the very Fleming--Viot diffusion semigroup on $C(\mathcal{P}(E))$
generated by the linear operator $\mathfrak{A}$. 
This limit theorem seems to be whole new and fascinating 
in that a measure-valued diffusion process 
can be obtained through the limit of the iterate of 
a certain positive linear operator as both the number of iteration 
and the dimension tend to infinity simultaneously 
under several assumptions.

\subsection{Organization}
    The rest of the present paper is organized as follows: 
    In Section \ref{Sect:Settings}, we give several properties of the
    Bernstein operator $B_{d,n}^{(\mathbf{q}_n)}$ associated with 
    $\mathbf{q}_n$ under Assumption {\bf (Q1)}. 
    All of our main results 
    in the present paper are stated in Section \ref{Sect:main results}. 
    We are going to show these results step by step, that is, 

    \vspace{1mm}
    \noindent
    {\bf Step 1}. limit theorems for the iterate of $B_{d,n}$ 
    (Theorems \ref{Thm:Bernstein-1} and \ref{Thm:Bernstein-2}),

    \vspace{1mm}
    \noindent
    {\bf Step 2}. limit theorems for the iterate of $B_{d,n}^{(\mathbf{q}_n)}$ 
    (Theorems \ref{Thm:Bernstein-general-1} and \ref{Thm:Bernstein-general-2}),

    \vspace{1mm}
    \noindent
    {\bf Step 3}. limit theorems for the iterate of $B_{d_n,n}^{(\mathbf{q}_n)}$ 
    (Theorem \ref{Thm:Bernstein-infinity-1}).

    \vspace{1mm}
    \noindent
    Sections \ref{Sect:Proofs-1} and \ref{Sect:Proofs-2}
    are devoted to the proofs of {\bf Step 1} and {\bf Step 2}, respectively. 
    The former part of Section \ref{Sect:Proofs-1} 
    (resp.~Section \ref{Sect:Proofs-2}) discusses the uniform convergence of 
    $\{n(B_{d,n}-I)\}_{n=1}^\infty$ 
    (resp.~$\{n(B_{d,n}^{(\mathbf{q}_n)}-I)\}_{n=1}^\infty$) 
    to the infinitesimal generator corresponding to 
    the stochastic differential equation \eqref{SDE:d-Bernstein} 
    (resp.~\eqref{SDE:d-Bernstein-q}), where $I$ denotes the identity operator. 
    See Lemma \ref{Lem:Voronovskaya-Berstein} 
    (resp.~Lemma \ref{Lem:Voronovskaya-Berstein-genreal}) for the statement. 
    This kind of convergence result is called the
    {\it Voronovskaya-type theorem} in the context of approximation theory. 
    Moreover, under several additional assumptions, 
    its convergence rate is also established, 
    which plays a key role in obtaining the rate of convergence for semigroups. 
    In the latter part of Section \ref{Sect:Proofs-1} 
    (resp.~Section \ref{Sect:Proofs-2}), 
    we show the uniform convergence of the iterate 
    $(B_{d,n})^{\lfloor nt \rfloor}$ 
    (resp.~$(B_{d,n}^{(\mathbf{q}_n)})^{\lfloor nt \rfloor}$) 
    to the Wright--Fisher diffusion semigroup corresponding to 
    \eqref{SDE:d-Bernstein} (resp.~\eqref{SDE:d-Bernstein-q}) together with 
    its rate of convergence, by applying the Voronovskaya-type theorem. 
    We show the convergence in law to the sequence of 
    linearly interpolated Markov chain generated 
    by the iterate of the Bernstein operator in the H\"older space as well.
    In Section \ref{Sect:Fleming--Viot}, 
    we give our infinite-dimensional framework and mention that several 
    classical examples of measure-valued diffusions 
    such as {\it Ohta--Kimura model} (cf.~\cite{OK}) 
    and {\it infinitely-many-neutral-alleles model with uniform mutation} (cf.~\cite{KC}) 
    are to be in scope of our framework 
    (see Examples \ref{Ex:Ohta-Kimura} and \ref{Ex:Kimura-Crow}).
    Furthermore, we establish the 
    Voronovskaya-type theorem (Lemma \ref{Thm:Voronovskaya-infinite})
    and the uniform convergence of 
    $(B_{d_n,n}^{(\mathbf{q}_n)})^{\lfloor nt \rfloor}$ 
    to the Fleming--Viot diffusion semigroup generated 
    by $\mathfrak{A}$ (Theorem \ref{Thm:Bernstein-infinity-1}). 
    Some conclusions and further possible directions 
    of this study are given in Section \ref{Sect:Conclusion}.

\subsection{Notations}

    We here fix several notations for later use. 
    For $\bx=(x_1, x_2, \dots, x_d) \in \R^d$, we write 
    \[
    |\bx|:=x_1+x_2+\cdots+x_d, \qquad 
    \|\bx\| := (x_1^2+x_2^2+\cdots+x_d^2)^{1/2}.
    \]
    Moreover, we sometimes write 
    \[
    \del_i=\frac{\del}{\del x_i}, \qquad 
    \del_{ij}=\frac{\del^2}{\del x_i \del x_j}, \qquad 
    \del_{ijk}=\frac{\del^3}{\del x_i \del x_j \del x_k},
    \qquad i, j, k=1, 2, \dots, d,
    \]
    in short. 
    For $1 \le r \le \infty$, 
    we denote by $C^r(\Delta_{d-1})$ 
    the space of all continuous functions defined on $ \Delta_{d-1} $
    having continuous partial derivatives of
    order $r$. When a function 
    $f : \Delta_{d-1} \to \mathbb{R}$ is Lipschitz continuous, 
    we put 
    \[
    \mathrm{Lip}(f):=\sup_{\bx, \by \in \Delta_{d-1}, \bx \neq \by}
    \frac{|f(\bx)-f(\by)|}{\|\bx-\by\|}(<\infty).
    \]
    For a multi-index $\bk=(k_1, k_2, \dots, k_d) \in \N_0^d$, 
    we put $|\bk|:=k_1+k_2+\cdots+k_d$. 
    The symbol $\lfloor x \rfloor$ is the floor function, that is, 
    $\lfloor x \rfloor:=\max\{n \in \mathbb{Z} \, | \, n \le x\}$ for $x \in \R$. For $a, b \in \R$, we write $a \vee b=\max\{a, b\}$. 
    Throughout the present paper, we denote by $C$ a 
    constant which may vary from line to line.

\section{{\bf The multidimensional Bernstein operator and its extension}}
\label{Sect:Settings}

    In this section, we introduce a multidimensional Bernstein operator $B_{d,n}$ 
    and give its several properties. 
    Furthermore, an extension of 
    the multidimensional Bernstein operator $B_{d,n}^{(\mathbf{q}_n)}$ is also 
    introduced, whose iterates yield more complex phenomena 
    in population genetics 
    as is seen in Section \ref{Sect:main results}.

\subsection{The multidimensional Bernstein operator on the simplex}

Similarly to the one-dimensional case, 
the $d$-dimensional Bernstein operator \eqref{Eq:d-Bernstein} 
can be represented as 
the expectation of a certain $d$-dimensional 
random variable. 
Let $\bx=(x_1, x_2, \dots, x_d) \in \Delta_{d-1}$ and $n \in \N$. 
Suppose that $S_n^i(x_i), \, i=1, 2, \dots, d$, is 
the binomial random variable with parameters $n$ and $x_i$. 
Then, the random variable 
$S_n(\bx)=(S_n^1(x_1), S_n^2(x_2), \dots, S_n^d(x_d))$ follows the
{\it multinomial distribution}, that is,  
    \[
    \mathbb{P}(S_n(\bx)=\bk)
    =\frac{n!}{k_1!k_2!\cdots k_d!}
    x_1^{k_1}x_2^{k_2}\cdots x_d^{k_d}, \qquad 
    \bx \in \Delta_{d-1}, \, \bk \in \N_0^d, \, |\bk| = n.
    \]
Then, we have the multidimensional version of \eqref{Bernstein-expectation}
as in the following:
    \begin{equation}
    \label{Eq:d-Bernstein-expectation}
    B_{d,n}f(\bx)=\mathbb{E}\left[ f\left(\frac{1}{n}S_n(\bx)\right) \right], 
    \qquad n \in \N, \, f \in C(\Delta_{d-1}), \, \bx \in \Delta_{d-1}. 
    \end{equation}
We now define a random function 
$G_n : \Delta_{d-1} \to \R$ by 
    \[
    G_n(\bx):=\frac{1}{n}S_n(\bx), \qquad \bx \in \Delta_{d-1}. 
    \]
Let $\{G_n^k\}_{k=1}^\infty$ be a sequence of 
independent copies of $G_n$. 
For $N=1, 2, 3, \dots$, we also define a random function
$H_n^N : \Delta_{d-1} \to \R$ by 
$H_n^N:=G_n^N \circ G_n^{N-1} \circ \cdots \circ G_n^1$.
Then, one  sees that 
    \begin{equation}\label{Eq:d-Bernstein-iterates}
    (B_{d,n})^Nf(\bx)
    =\mathbb{E}\left[f\big(H_n^N(\bx)\big)\right],
    \qquad N=0, 1, 2, \dots, \, \bx \in \Delta_{d-1},
    \end{equation}
with the convention that $H_n^0(\x)=\x$. 
Moreover, the sequence $\{H_n^N(\bx)\}_{N=1}^\infty$
is a time-homogeneous Markov chain taking values with
    \[
    \mathcal{I}=\left\{\frac{\bi}{n}=\left(\frac{i_1}{n}, \frac{i_2}{n}, 
    \dots, \frac{i_d}{n}\right) \, \Big| \, 
    i_\ell=0, 1, 2, \dots, n, \,\, 
    \ell=1, 2, \dots, d, \,\, |\bi|=\sum_{\ell=1}^d i_\ell = n
    \right\}.
    \]
Its one-step transition probability is given by 
    \begin{equation}\label{Eq:transition-prob-Bernstein}
    \begin{aligned}
    &\mathbb{P}\left(
    H_n^{N+1}(\bx)=\frac{\bj}{n} \, \Big| \, 
    H_n^{N}(\bx)=\frac{\bi}{n}\right) 
    =\frac{n!}{j_1!j_2!\cdots j_d!}\left(\frac{i_1}{n}\right)^{j_1}
    \left(\frac{i_2}{n}\right)^{j_2}\cdots
    \left(\frac{i_d}{n}\right)^{j_d}
    \end{aligned}
    \end{equation}
 for all $\bi, \bj \in \N_0^d$ with 
 $\bi/n, \, \bj/n \in \mathcal{I}$. 
 This means that the Markov chain $\{nH_n^N\}_{N=1}^\infty$
 is viewed as the $d$-{\it allele Wright--Fisher model 
 without mutation}, which describes
 a population of individuals of different types.
 We refer to e.g., \cite{Feller} 
 for an early work on some relations 
 between the Wright--Fisher model
 and certain diffusion phenomena, and 
 \cite{EK} for a comprehensive 
 explanation of the Wright--Fisher model from mathematical perspectives.

\begin{re}\normalfont
The $d$-dimensional Bernstein operator 
on the $d$-dimensional hypercube $[0, 1]^d$ is also known. 
This is defined by 
    \[
    \begin{aligned}
    \ol{B}_{d, n}f(\bx)&=\sum_{\bk \in \N_0^d, |\bk| = n}
    \binom{n}{k_1}\binom{n}{k_2} \cdots \binom{n}{k_d} \\
    &\hspace{1cm}\times 
    x_1^{k_1}(1-x_1)^{n-k_1}x_2^{k_2}(1-x_2)^{n-k_2}
    \cdots x_d^{k_d}(1-x_d)^{n-k_d}
     \end{aligned}
     \]
for $f \in C([0, 1]^d)$ and $\bx \in [0, 1]^d$. 
See e.g., \cite{Butzer} for an early work
about this operator. 
It is known that 
the linear operator $\ol{B}_{d,n}$ also has the 
same approximation property as Proposition \ref{prop:approximation-d}.
(see e.g., \cite[Theorem 4.11]{Altomare}). 
We do not discuss this operator in the present paper.  
However, it might be an interesting problem 
to establish limit theorems for its iterates
in a probabilistic point of view. 
\end{re}

\subsection{An extension  of the multidimensional Bernstein operator}

As  in the one-dimensional case, 
the iterate of the $d$-dimensional Bernstein operator 
is expected to converge to the $d$-allele Wright--Fisher model without mutation.
In order to create some effects of mutation in 
the limiting phenomena, we take $\mathbf{q}_n=\{q_{ij}^{(n)}\}_{i, j=1}^d$ to be 
real numbers satisfying \eqref{Eq:cond-q} and introduce the notion of 
the $d$-dimensional Bernstein operator $B_{d, n}^{(\mathbf{q}_n)}$ 
which is associated with the
intensity of the mutation from allele $i$ to 
allele $j$ for $i, j=1, 2, \dots, d$ with $i \neq j$.
See Definition \ref{Def:modified Bernstein} for more details.

Then, one wonders if the Bernstein operator $B_{d, n}^{(\mathbf{q}_n)}$
also approximates every continuous function in $C(\Delta_{d-1})$ uniformly or not. 
To deduce the approximation property, we need to put the following assumption for $\mathbf{q}_n$.

\vspace{2mm}
\noindent
{\bf (Q1)}: There exist real numbers $\mathbf{q}=\{q_{ij}\}_{i, j=1}^d$ satisfying 
\begin{itemize}
    \item 
    It holds that $q_{ij}>0, \, i \neq j$, and 
    \[
    q_{ii}=-\sum_{j \neq i}q_{ij}, \quad i=1, 2, \dots, d.
    \]
    \item 
    There exists some $C>0$ independent of $n \in \N$ such that
    $$
    \left|q_{ij}^{(n)}-\frac{1}{n}q_{ij}\right| \le \frac{C}{n^\gamma}, \qquad i, j=1, 2, \dots, d, \, n \in \N,
    $$
    for some $\gamma>1$.
\end{itemize}
  
\noindent
Then, under Assumption {\bf (Q1)}, we have the following approximation property for $B_{d, n}^{(\mathbf{q}_n)}$
similarly to the case of $B_{d, n}$.
\begin{pr}
    \label{prop:approximation-d-modified}
    We assume {\bf (Q1)}. Then, for any $f \in C(\Delta_{d-1})$, we have
    \[
    \lim_{n \to \infty}\|B_{d,n}^{(\mathbf{q}_n)}f-f\|_\infty=0. 
    \]
    
\end{pr}

\begin{proof}
    Let $\bm{1}$ be the constant function 
    $\bm{1}(\bx) \equiv 1$ 
    and $e_i(\bx)=x_i$ for $\bx=(x_1, x_2, \dots, x_d) \in \Delta_{d-1}$ and $i=1, 2, \dots, d$.
    It follows from the multinomial theorem 
    that $B_{d, n}^{(\mathbf{q}_n)}\bm{1}=\bm{1}$. 
    Moreover, we have 
    \[
\begin{aligned}
    B_{d, n}^{(\mathbf{q}_n)}e_i(\bx)
    &= \sum_{\bk \in \N_0^d, |\bk| = n} 
            \frac{n!}{k_1!k_2!\cdots k_d!}
            (x_1^{(\mathbf{q}_n)})^{k_1}(x_2^{(\mathbf{q}_n)})^{k_2}\cdots 
            (x_d^{(\mathbf{q}_n)})^{k_d}
             \times \frac{k_i}{n} \\
    &= x_i^{(\mathbf{q}_n)}\sum_{\substack{\bk \in \N_0^d, k_i \ge 1 \\ |\bk| = n-1}}
            \frac{(n-1)!}{k_1! \cdots (k_i-1)! \cdots k_d!}
            (x_1^{(\mathbf{q}_n)})^{k_1}\cdots
            (x_i^{(\mathbf{q}_n)})^{k_i-1}\cdots 
            (x_d^{(\mathbf{q}_n)})^{k_d} \\
    &= x_i^{(\mathbf{q}_n)}, \qquad i=1, 2, \dots, n,
\end{aligned}
    \]
and 
    \[
\begin{aligned}
    B_{d, n}^{(\mathbf{q}_n)}e_i^2(\bx)
    &= \sum_{\bk \in \N_0^d, |\bk| = n} 
            \frac{n!}{k_1!k_2!\cdots k_d!}
            (x_1^{(\mathbf{q}_n)})^{k_1}(x_2^{(\mathbf{q}_n)})^{k_2}\cdots 
            (x_d^{(\mathbf{q}_n)})^{k_d}
             \times \frac{k_i^2}{n^2} \\
    &= \frac{n-1}{n}x_i^{(\mathbf{q}_n)}
    \sum_{\substack{\bk \in \N_0^d, k_i \ge 1 \\ |\bk| = n-1}} 
    \frac{(k_i-1)+1}{n-1} \times 
            \frac{(n-1)!}{k_1!\cdots (k_i-1)!\cdots k_d!}\\
            &\hspace{1cm}\times 
            (x_1^{(\mathbf{q}_n)})^{k_1}\cdots
            (x_i^{(\mathbf{q}_n)})^{k_i-1}\cdots 
            (x_d^{(\mathbf{q}_n)})^{k_d} \\
    &= \frac{n-1}{n}(x_i^{(\mathbf{q}_n)})^2
    \sum_{\substack{\bk \in \N_0^d, k_i \ge 2 \\ |\bk| = n-2}}
            \frac{(n-2)!}{k_1! \cdots (k_i-2)! \cdots k_d!}\\
            &\hspace{1cm}\times
            (x_1^{(\mathbf{q}_n)})^{k_1}\cdots
            (x_i^{(\mathbf{q}_n)})^{k_i-2}\cdots 
            (x_d^{(\mathbf{q}_n)})^{k_d} 
            +\frac{1}{n}x_i^{(\mathbf{q}_n)}\\
    &= \frac{n-1}{n}(x_i^{(\mathbf{q}_n)})^2
    +\frac{1}{n}x_i^{(\mathbf{q}_n)}, \qquad i=1, 2, \dots, n.
\end{aligned}
    \]
Then, we obtain 
    \[
    \lim_{n \to \infty}
    \|B_{d, n}^{(\mathbf{q}_n)}e_i - e_i\|_\infty=0, 
    \qquad \lim_{n \to \infty}
    \|B_{d, n}^{(\mathbf{q}_n)}e_i^2 - e_i^2\|_\infty=0, 
    \qquad i=1, 2, \dots, d,
    \]
by using Assumption {\bf (Q1)}. 
Therefore, Korovkin's theorem 
(cf.~\cite[Theorem 4.2]{Altomare}) immediately implies 
the desired uniform convergence. 
\end{proof}

Let $\bx=(x_1, x_2, \dots, x_d) \in \Delta_{d-1}$ and $n \in \N$. 
Suppose that $S_n^i(x_i^{(\mathbf{q}_n)}), \, i=1, 2, \dots, d$, is 
the binomial random variable with parameters $n$ and $x_i^{(\mathbf{q}_n)}$. 
Then, the distribution of the random variable 
$S_n(\bx^{(\mathbf{q}_n)})=\big(S_n^1(x_1^{(\mathbf{q}_n)}), S_n^2(x_2^{(\mathbf{q}_n)}), \dots, S_n^d(x_d^{(\mathbf{q}_n)})\big)$ is given by  
    \[
    \mathbb{P}(S_n(\bx^{(\mathbf{q}_n)})=\bk)
    =\frac{n!}{k_1!k_2!\cdots k_d!}
    (x_1^{(\mathbf{q}_n)})^{k_1}
    (x_2^{(\mathbf{q}_n)})^{k_2}\cdots 
    (x_d^{(\mathbf{q}_n)})^{k_d}
    \]
for $\bx \in \Delta_{d-1}$ and $\bk \in \N_0^d$ with  $|\bk| = n$.
In the same manner as \eqref{Eq:d-Bernstein-iterates}, 
we can also associate the iterate of $B_{d, n}^{(\mathbf{q}_n)}$
with an $\mathcal{I}$-valued  time-homogeneous Markov chain 
$\{(H_n^{(\mathbf{q}_n)})^N(\bx)\}_{N=0}^\infty$
with $(H_n^{(\mathbf{q}_n)})^0(\x)=\x$. 
Namely, it holds that
\[
(B_{d,n}^{(\mathbf{q}_n)})^Nf(\bx)=\mathbb{E}\left[f\big((H_n^{(\mathbf{q}_n)})^N(\bx)\big)\right],
    \qquad N=0, 1, 2, \dots, \, \bx \in \Delta_{d-1}.
 \]
The one-step transition probability 
of $\{(H_n^{(\mathbf{q}_n)})^N(\bx)\}_{N=0}^\infty$ is 
written as
    \begin{align}\label{Eq:transition-prob-generalized-Bernstein}
    &\mathbb{P}\left(
    (H_n^{(\mathbf{q}_n)})^{N+1}(\bx)=\frac{\bj}{n} \, 
    \Big| \, 
    (H_n^{(\mathbf{q}_n)})^{N}(\bx)=\frac{\bi}{n}\right)\notag \\
    &=\frac{n!}{j_1!j_2!\cdots j_d!}
    \left(
    \frac{i_1}{n}+
    \sum_{j=1}^d q_{j1}^{(n)}\frac{i_j}{n}\right)^{j_1}
    \left(
    \frac{i_2}{n}+
    \sum_{j=1}^d q_{j2}^{(n)}\frac{i_j}{n}\right)^{j_2}
    \cdots
    \left(
    \frac{i_d}{n}+
    \sum_{j=1}^d q_{jd}^{(n)}\frac{i_j}{n}\right)^{j_d}\nn
    \end{align}
 for all $\bi, \bj \in \N_0^d$ with 
 $\bi/n, \, \bj/n \in \mathcal{I}$.

\subsection{Martingale property}

It is known that both  Markov chains 
$\{H_n^N(\bx)\}_{N=0}^\infty$ 
and $\{(H_n^{(\mathbf{q}_n)})^N(\bx)\}_{N=0}^\infty$ 
possess the martingale property, as in the following proposition. 
This property plays an important role in showing our main results, 
in particular, the tightness of the probability measures on a path space
induced by these Markov chains. 
See the proofs of 
Theorems \ref{Thm:Bernstein-2} and \ref{Thm:Bernstein-general-2}.

\begin{pr}
 \label{prop:martingale}
{\rm (1)} For every $\bx \in \Delta_{d-1}$, 
 the Markov chain $\{H_n^N(\bx)\}_{N=0}^\infty$
 is an $\mathcal{I}$-valued martingale. 

 \vspace{2mm}
 \noindent
 {\rm (2)} For every $\bx \in \Delta_{d-1}$, 
 the Markov chain 
     \[
    \{(H_n^{(\mathbf{q}_n)})^N(\bx) - (\bx^{(\mathbf{q}_n)}-\bx)\}_{N=0}^\infty
    \]
 is an $\mathcal{I}$-valued martingale. 
 
\end{pr}
    
\begin{proof}
Fix $\x \in \Delta_{d-1}$.
By definition, the summability of each $H_n^N(\x)$ is clear. 
Moreover, for $N \in \N$ and a state $\mathbf{i}/n \in \mathcal{I}$, 
it follows from \eqref{Eq:transition-prob-Bernstein} that 
\begin{align*}
    &\E \Big[ H_n^{N+1}(\x) \, \Big| \, H_n^N(\x)=\frac{\mathbf{i}}{n}\Big] \\
    &=\sum_{\mathbf{j}/n \in \mathcal{I}} \frac{\mathbf{j}}{n} \times 
    \frac{n!}{j_1!j_2!\cdots j_d!}\left(\frac{i_1}{n}\right)^{j_1}
    \left(\frac{i_2}{n}\right)^{j_2}\cdots
    \left(\frac{i_d}{n}\right)^{j_d}
    =\frac{1}{n} \times n \times \frac{\mathbf{i}}{n}=\frac{\mathbf{i}}{n}=H_n^N(\x),
\end{align*}
which implies that Assertion (1) is true. 
Assertion (2) is also shown in a quite similar way to
the above. 
\end{proof}

\section{{\bf Main results}} 
\label{Sect:main results}

We are in a position to state the main results of the 
present paper. All of the proofs are demonstrated 
in the subsequent sections.

\subsection{Limit theorems for iterates of $B_{d, n}$}

At the beginning, we aim to find out 
the limiting behaviors of the iterate
of the $d$-dimensional 
Bernstein operator $B_{d, n}$. 
Let $\mathcal{A}_d$ be the second order differential operator 
acting on $C^2(\Delta_{d-1})$ defined by 
    \begin{equation}\label{Eq:generator-Bernstein}
    \mathcal{A}_df(\bx):=
    \frac{1}{2}\sum_{i, j=1}^d x_i
    (\delta_{ij}-x_j)\frac{\del^2f}{\del x_i \del x_j}(\bx)
    \end{equation}
for $f \in C^2(\Delta_{d-1})$ and $\bx \in \Delta_{d-1}$,
where $\delta_{ij}$ stands for the usual Kronecker's delta. 
Then, the following is well known. 

\begin{lm}[see e.g., \cite{Alt-C, Ethier}]
\label{Lem:Altomare-Campiti}
The closure of $\big(\mathcal{A}_d, C^2(\Delta_{d-1})\big)$
generates a contraction $C_0$-semigroup 
$({\sf T}_t)_{t \ge 0}$ 
on $C(\Delta_{d-1})$
and $C^2(\Delta_{d-1})$ is a core for the closure. 
\end{lm}

We can show that the iterate of the $d$-dimensional 
Bernstein operator $B_{d,n}$ uniformly converges to
the contrction $C_0$-semigroup $({\sf T}_t)_{t \ge 0}$ 
generated by $\mathcal{A}_d$ together with 
its rate of convergence. 
The following is our first main results. 

\begin{tm}
\label{Thm:Bernstein-1}
        For $ \bx \in \Delta_{d-1}, $
        let 
        $\big(\mathsf{X}_t(\bx)\big)_{t \ge 0}=
        \big(\mathsf{X}_t^1(\bx), \mathsf{X}_t^2(\bx), 
        \dots, \mathsf{X}_t^d(\bx)\big)_{t \ge 0}$ be the 
        $d$-dimensional diffusion process
        which solves the stochastic differential equation \eqref{SDE:d-Bernstein}.
        Then, we have 
        \begin{equation}\label{Eq:Bernstein-coinidence}
        {\sf T}_t f(\x)=\E\big[f\big({\sf X}_t(\x)\big)\big], \qquad f \in C(\Delta_{d-1}), \, \x \in \Delta_{d-1}, \, t \ge 0,
        \end{equation}
        and 
        \begin{equation*}
            \lim_{n \to \infty}
            \left\|(B_{d,n})^{\lfloor nt \rfloor}f 
            - \E\big[f\big({\sf X}_t(\cdot)\big)\big]\right\|_\infty=0
            , \qquad f \in C(\Delta_{d-1}), \, t \ge 0.
        \end{equation*}
        Moreover, if $f \in C^2(\Delta_{d-1})$ satisfies that
        
        \vspace{1mm}
        \noindent
        {\bf (A1):} ${\sf T}_t f \in C^2(\Delta_{d-1}), \, t \ge 0$, and
        
        \vspace{1mm}
        \noindent
        {\bf (A2):} 
        $\del_{ij}f$ and $\del_{ij}({\sf T}_t f)$, $i, j=1, 2, \dots, d$, 
        are Lipschitz continuous on $\Delta_{d-1}$, 
        
        \vspace{1mm}
        \noindent
        then we have 
        \begin{align}
        &\|(B_{d,n})^{\lfloor nt \rfloor} f - \E\left[f\big({\sf X}_t(\cdot)\big)\right]\|_{\infty} \nonumber\\
        &\le \left(\frac{t}{\sqrt{n}} + \frac{1}{n}\right)
        \left(\|\A_d f\|_{\infty} + \frac{1}{\sqrt{n}}\left(\frac{d^{5/2}}{16 \cdot 3^{1/4}}
        \max_{i, j=1, 2, \dots, d}
        \mathrm{Lip}\left(\del_{ij}f\right) \right)
        \right)\nonumber\\
        &\hspace{1cm} + 
        \frac{1}{\sqrt{n}}\int^t_0 \left(\frac{d^{5/2}}{16 \cdot 3^{1/4}}
        \max_{i, j=1, 2, \dots, d}
        \mathrm{Lip}\big(\del_{ij} ({\sf T}_s f)\big) \right)
         \, \dd s, \qquad n \in \N, \,\, t \ge 0.
         \label{Eq:Bernstein-rate}
    \end{align}
\end{tm}

We denote by $C_{\bx}([0, 1]; \, \R^d)$ 
the set of all continuous paths
$w : [0, 1] \to \R^d$ satisfying $w(0)=\bx \in \Delta_{d-1}$
with the uniform topology. 
For $\alpha \in (0, 1)$, we also denote by 
$C_{\bx}^{\alpha\text{{\rm -H\"ol}}}([0, 1]; \, \R^d)$
the set of all $\alpha$-H\"older continuous paths 
$w : [0, 1] \to \R^d$ with $w(0)=\bx \in \Delta_{d-1}$.
Let us consider the random curve defined by 
$t \longmapsto H_n^{\lfloor nt \rfloor}(\bx)$ for $\bx \in \Delta_{d-1}$. 
We then set 
\begin{equation}
\label{LI:Bernstein}
    \mathcal{H}_t^{(n)}(\bx)
    := 
    H^{\lfloor nt \rfloor}_n(\x) 
    + (n t - \lfloor nt \rfloor)
    \left(H^{\lfloor nt \rfloor + 1}_n(\x) 
    - H^{\lfloor nt \rfloor}_n (\x) \right)
\end{equation}
for $0 \le t \le 1$ and $\bx \in \Delta_{d-1}$, that is, 
each $\mathcal{H}_\cdot^{(n)}(\bx)$ 
is a $C_{\bx}([0, 1]; \, \R^d)$-valued 
random variable
constructed via the linear interpolation of 
the Markov chain
$\{H_n^N\}_{N=0}^\infty$.

The following is the functional limit theorem for 
the Markov chain induced by the $d$-dimensional
Bernstein operator $B_{d,n}$, which gives a much stronger convergence
than that obtained in Theorem \ref{Thm:Bernstein-1}.

\begin{tm}
\label{Thm:Bernstein-2}
The sequence $\big(\mathcal{H}_t^{(n)}(\bx)\big)_{0 \le t \le 1}$, $n=1, 2, 3, \dots$, 
of continuous stochastic processes 
converges in law to the diffusion process 
$\big(\mathsf{X}_t(\bx)\big)_{0 \le t \le 1}$ 
in $C_{\bx}^{\alpha\text{{\rm -H\"ol}}}([0, 1], \R^d)$
for all $\alpha<1/2$. 
\end{tm}

\subsection{Limit theorems for iterates of $B_{d, n}^{(\mathbf{q}_n)}$}

Let $\mathbf{q}_n=\{q_{ij}^{(n)}\}_{i,j=1}^d$ be real numbers 
satisfying {\bf (Q1)}. 
We next present 
the limit theorems for the iterates
of \eqref{Eq:q-Bernstein} associated with $\mathbf{q}_n$. 
Let $\mathcal{A}_d^{(\mathbf{q})}$ be the second order differential operator 
acting on $C^2(\Delta_{d-1})$ defined by 
    \begin{equation}\label{Eq:generator-Bernstein-general}
    \mathcal{A}_d^{(\mathbf{q})}f(\bx):=
    \frac{1}{2}\sum_{i, j=1}^d x_i
    (\delta_{ij}-x_j)\frac{\del^2f}{\del x_i \del x_j}(\bx)
    +\sum_{i=1}^d \left(\sum_{j=1}^d q_{ji}x_j\right)\frac{\del f}{\del x_i}
    \end{equation}
for $f \in C^2(\Delta_{d-1})$ and $\bx \in \Delta_{d-1}$, 
where $\mathbf{q}=\{q_{ij}\}_{i,j=1}^d$ is a set of real numbers as in Assumption {\bf (Q1)}. 
We should note that 
some extensions of this kind of differential operators
have already been discussed in e.g.,~\cite{ACD}. 
Particularly, we have the following property 
for our operator $\A_d^{(\mathbf{q})}$, similarly to Lemma \ref{Lem:Altomare-Campiti}.

\begin{lm}[see {\cite[Theorem 4.1]{ACD}}]
\label{Lem:core-general}
The closure of 
$\big(\mathcal{A}_d^{(\mathbf{q})}, C^2(\Delta_{d-1})\big)$
generates a contraction $C_0$-semigroup 
$({\sf T}^{(\mathbf{q})}_t)_{t \ge 0}$ 
on $C(\Delta_{d-1})$
and $C^2(\Delta_{d-1})$ is a core for the closure. 
\end{lm}

The following is a generalization of Theorem \ref{Thm:Bernstein-1}
to the case where we take the effect of mutations
among alleles into account.

\begin{tm}
\label{Thm:Bernstein-general-1}
        We assume {\bf (Q1)}.
        For $ \bx \in \Delta_{d-1}, $
        let 
        \[ 
        \big(\mathsf{X}_t^{(\mathbf{q})}(\bx)\big)_{t \ge 0}=
        \big(\mathsf{X}_t^{(\mathbf{q}), 1}(\bx), 
        \mathsf{X}_t^{(\mathbf{q}), 2}(\bx), 
        \dots, \mathsf{X}_t^{(\mathbf{q}), d}(\bx)\big)_{t \ge 0}
        \]
        be the 
        $d$-dimensional diffusion process
        which solves the stochastic differential equation
        \begin{equation*}
        \label{SDE:d-Bernstein-general}
            \dd \mathsf{X}_t^{(\mathbf{q}), i}(\bx)=\sum_{j=1}^d
            \sigma_{ij}\big(\mathsf{X}_t^{(\mathbf{q})}(\bx)\big) \, \dd W_t^j
            + 
            \sum_{j=1}^d q_{j i} \mathsf{X}_t^{(\mathbf{q}), j}(\bx) 
            \, \dd t,
        \end{equation*}
        for $i=1, 2, \dots, d$, with 
        $\mathsf{X}_0^{(\mathbf{q})}(\bx)=\bx$, 
        where $\sigma=(\sigma_{ij})_{i, j=1}^d : \Delta_{d-1} \to \R^d \otimes \R^d$ is defined by \eqref{Eq:diffusion-matrix} and 
        $ (W_t)_{t \ge 0}=(W_t^1, W_t^2, \dots, W_t^d)_{t \ge 0} $ 
        is a $d$-dimensional standard Brownian motion. 
        Then, we have 
        \begin{equation*}
        {\sf T}_t^{(\mathbf{q})} f(\x)
        =\E\big[f\big({\sf X}_t^{(\mathbf{q})}(\x)\big)\big], 
        \qquad f \in C(\Delta_{d-1}), \, 
        \x \in \Delta_{d-1}, \, t \ge 0,
        \end{equation*}
        and 
        \begin{equation}
            \lim_{n \to \infty}
            \left\|(B_{d,n}^{(\mathbf{q}_n)})^{\lfloor nt \rfloor}f 
            - \E\big[f\big({\sf X}_t^{(\mathbf{q})}(\cdot)\big)\big]\right\|_\infty=0
            , \qquad f \in C(\Delta_{d-1}), \, t \ge 0.
            \label{Eq:Bernstein-convergence-general}
        \end{equation}
   Moreover, if $f \in C^2(\Delta_{d-1})$ satisfies that
        
        \vspace{1mm}
        \noindent
        {\bf (A3):} ${\sf T}_t^{(\mathbf{q})} f \in C^2(\Delta_{d-1}), \, t \ge 0$, and
        
        \vspace{1mm}
        \noindent
        {\bf (A4):} 
        $\del_{ij}f$ and $\del_{ij}({\sf T}_t^{(\mathbf{q})} f)$, $i, j=1, 2, \dots, d$, 
        are Lipschitz continuous on $\Delta_{d-1}$, 

        \vspace{1mm}
        \noindent
        then we have 
        \begin{align}
        &\left\|(B_{d,n}^{(\mathbf{q}_n)})^{\lfloor nt \rfloor}f 
            - \E\big[f\big({\sf X}_t^{(\mathbf{q})}(\cdot)\big)\big]\right\|_\infty
        \le \frac{C}{\sqrt{n}},
         \qquad n \in \N, \,\, t \ge 0,
         \label{Eq:Bernstein-rate-general}
    \end{align}
    for some $C>0$ which depends on 
    $d$, $t$, $\|\mathcal{A}_d^{(\mathbf{q})}f\|_\infty$, $\|\del_i f\|_\infty$, 
    $\|\del_{ij}f\|_\infty$, 
    $\mathrm{Lip}(\del_{ij}f)$ 
    and 
    $\mathrm{Lip}(\del_{ij}{\sf T}_t^{(\mathbf{q})}(f))$ 
    for 
    $i, j=1, 2, \dots, d$. 
\end{tm}    

\noindent
Note that the positive constant $C>0$ in the right-hand-side of 
\eqref{Eq:Bernstein-rate-general} is no longer written down explicitly 
since we need to use Assumption {\bf (Q1)} to give the estimate.

In the same manner as \eqref{LI:Bernstein}, we define
\begin{align}
    (\mathcal{H}_t^{(\mathbf{q_n})})^{(n)}(\bx)
    &:= 
    (H_n^{(\mathbf{q}_n)})^{\lfloor nt \rfloor}(\x) 
    + (n t - \lfloor nt \rfloor) \notag \\
    &\hspace{1cm}\times 
    \left\{(H_n^{(\mathbf{q}_n)})^{\lfloor nt \rfloor + 1}(\x) 
    - (H_n^{(\mathbf{q}_n)})^{\lfloor nt \rfloor} (\x) \right\}
    -nt(\bx^{(\mathbf{q}_n)}-\bx)
    \label{LI:Bernstein-general}
\end{align}
for $0 \le t \le 1$ and $\bx \in \Delta_{d-1}$, 
which gives a sequence of $C_{\bx}([0, 1]; \, \R^d)$-valued 
random variables 
constructed by 
the Markov chain
$\{(H_n^{(\mathbf{q}_n)})^N(\bx)\}_{N=0}^\infty$. 
Then, we also have the functional limit theorem
for the sequence 
$
\{(\mathcal{H}_\cdot^{(\mathbf{q}_n)})^{(n)}\}
_{n=1}^\infty$. 

\begin{tm}
\label{Thm:Bernstein-general-2}
The sequence $\big((\mathcal{H}_t^{(\mathbf{q}_n)})^{(n)}(\bx)\big)_{0 \le t \le 1}$, $n=1, 2, 3, \dots$, 
of continuous stochastic processes 
converges in law to the diffusion process 
$\big(\mathsf{X}_t^{(\mathbf{q})}(\bx)\big)_{0 \le t \le 1}$ 
in $C_{\bx}^{\alpha\text{{\rm -H\"ol}}}([0, 1], \R^d)$
for all $\alpha<1/2$. 
\end{tm}

\subsection{Limit theorems for iterates of $B_{d_n, n}^{(\mathbf{q}_n)}$
and the Fleming--Viot process}

Recall that $E$ is a compact metric space, 
which corresponds to the set of 
all possible allele types. 
Though $E$ may be locally compact 
in most applications, such cases 
can be reduced to the compact cases 
after all, by one-point compactification. 
Let $C(E)$ be the Banach space of all continuous functions on $E$
with the supremum norm $\|\cdot\|_\infty$. 
For $\beta \in C(E)$ and $\mu \in \mathcal{P}(E)$, we write 
\[
\langle \beta, \mu \rangle:=\int_E \beta(z) \, \mu(\mathrm{d}z). 
\]
We note that $\mathcal{P}(E)$ 
is metrizable and is a compact metric space. 
Denote by $C(\mathcal{P}(E))$
the Banach space of all continuous functions  
on $\mathcal{P}(E)$ with the supremum norm $\|\cdot\|_\infty$.

Let $(\mathcal{T}_t)_{t \ge 0}$ be the Feller semigroup on $C(E)$
generated by the mutation operator $\mathcal{Q}$ of \eqref{Eq:generator-FV} and 
$P(t, z, \dd \xi)$ its transition function. 
Then, for every $N \in \N$, $t \ge 0$ and $z_1, z_2, \dots, z_N \in E$, 
we also define the semigroup 
$(\mathcal{T}_t^{(N)})_{t \ge 0}$ on $C(E^N)$ by 
$$
\mathcal{T}_t^{(N)} \beta (z_1, z_2, \dots, z_N)
=\int_E \cdots \int_E \beta (\xi_1, \xi_2, \dots, \xi_N) 
\, P(t, z_1, \dd \xi_1) \cdots P(t, z_N, \dd \xi_N).
$$
We denote by $(\mathcal{Q}^{(N)}, \Dom(\mathcal{Q}^{(N)}))$
the infinitesimal generator of $(\mathcal{T}_t^{(N)})_{t \ge 0}$. 
Then, for $N \in \N$ with $N \ge 2$ and $1 \le \ell_1<\ell_2 \le N$, 
the {\it sampling operator} $\Phi_{\ell_1\ell_2}^{(N)} : C(E^N) \to C(E^{N-1})$
is defined by letting $\Phi_{\ell_1\ell_2}^{(N)}\beta$
be the function obtained from $\beta$ by replacing $z_{\ell_2}$ by $z_{\ell_1}$
and renumbering the variables, that is, 
$$
(\Phi_{\ell_1\ell_2}^{(N)}\beta)(z_1, z_2, \dots, z_{N-1})
:=\beta(z_1, \dots, z_{\ell_2-1}, z_{\ell_1}, z_{\ell_2+1}, \dots, z_N). 
$$
For example, when $N=3$, 
the sampling operators $\Phi_{\ell_1\ell_2}^{(3)}, \, (\ell_1, \ell_2)=(1, 2), (1, 3), (2, 3)$, 
are concretely given by 
$$
\begin{aligned}
\Phi_{12}^{(3)}\beta(z_1, z_2)=\beta(z_1, z_1, z_2), \quad 
\Phi_{13}^{(3)}\beta(z_1, z_2)=\beta(z_1, z_2, z_1), \quad
\Phi_{23}^{(3)}\beta(z_1, z_2)=\beta(z_1, z_2, z_2).
\end{aligned}
$$

We now put 
$$
\mathcal{D}=\{\varphi(\mu)=\la \beta, \mu^{\otimes N}\ra \mid \beta \in \mathrm{Dom}(\mathcal{Q}^{(k)})
\cap C(E^k), \, N \in \mathbb{N}\},
$$
where $\mu^{\otimes N}$ stands for the $N$-fold 
product probability measure of $\mu$ itself. 
Note that $\mathcal{D}$ is clearly a dense subset of $C(\mathcal{P}(E))$. 
Then, for $\varphi(\mu)=\la \beta, \mu^{\otimes N} \ra\in \mathcal{D}$, 
the linear operator $\mathfrak{A}$ defined by \eqref{Eq:generator-FV} is 
reduced to 
\begin{equation}\label{Eq:generator-FV-2}
\mathfrak{A}\varphi(\mu)
    =\sum_{1 \le \ell_1 < \ell_2 \le N}\Big(
    \la \Phi_{\ell_1\ell_2}^{(N)}\beta, \mu^{\otimes(N-1)}\ra - \la \beta, \mu^{\otimes N}\ra \Big)+\la \mathcal{Q}^{(N)}\beta, \mu^{\otimes N}\ra
\end{equation}
in terms of the sampling operator $\Phi_{\ell_1\ell_2}^{(N)}$. 
Ethier and Kurtz have shown in \cite{EK93} the following property. 

\begin{lm}[see {\cite[Theorem 3.4]{EK93}}]\label{Lem:core-infinity}
The closure of 
$(\mathfrak{A}, \mathcal{D})$ generates a Feller semigroup
on $C(\mathcal{P}(E))$ and $\mathcal{D}$ is a core for the closure.
\end{lm}

Let $E^{(d)}=\{z_1^{(d)}, z_2^{(d)}, \dots, z_d^{(d)}\} 
\subset E, \, d=2, 3, 4, \dots,$ be 
a sequence of finite subsets, 
where $z_i^{(d)} \neq z_j^{(d)}$ whenever $i \neq j$. 
Then, we easily see that 
the $(d-1)$-simplex $\Delta_{d-1}$ 
is always identified with a subset 
\[
\mathcal{P}(E^{(d)})=\left\{
\mu_{\bx}=\sum_{i=1}^d x_i \delta_{z_i^{(d)}} \, \Big| \, 
x_1, x_2, \dots, x_d \ge 0, \, \sum_{i=1}^d x_i=1\right\} \subset 
\mathcal{P}(E).
\]
The sequence $\{E^{(d)}\}_{d=2}^\infty$ is called 
a {\it discretization} of the compact metric space $E$. 
We define 
$\widehat{\pi}_d : \Delta_{d-1} \to \mathcal{P}(E)$ by
\[
\widehat{\pi}_d(\bx):=\mu_{\bx} =
\sum_{i=1}^d x_i \delta_{z_i^{(d)}}, \qquad 
\bx=(x_1, x_2, \dots, x_d) 
\in \Delta_{d-1},
\]
and 
$P_d : C(\mathcal{P}(E)) \to C(\Delta_{d-1})$ by 
\[
P_d\varphi(\bx)
:=(\varphi \circ \widehat{\pi}_d)(\bx)=\varphi(\mu_\bx), \qquad 
\varphi \in C(\mathcal{P}(E)), \, 
\bx \in \Delta_{d-1}.
\]
Since it holds that $\|P_d\varphi\|_\infty=\|\varphi\|_\infty$
as $d \to \infty$ for every $\varphi \in C(\mathcal{P}(E))$,
the sequence 
$\{(C(\Delta_{d-1}), P_d)\}_{d=2}^\infty$ approximates 
the Banach space $(C(\mathcal{P}(E)), \|\cdot\|_\infty)$
in the sense of Trotter \cite{Trotter}.

Let $\{d_n\}_{n=1}^\infty$ be an increasing sequence 
of positive integers strictly bigger than one
and $E^{(d_n)}=\{z_1^{(d_n)}, z_2^{(d_n)}, 
\dots, z_{d_n}^{(d_n)}\}$
a discretization of $E$ along the sequence 
$\{d_n\}_{n=1}^\infty$. 
In what follows, we always assume that 

\vspace{2mm}
\noindent
{\bf (D)}: it holds that $d_n=o(n^{1/8})$ as $n \to \infty$.

\vspace{2mm}
\noindent
Since
the dimension $d_n$ of the Bernstein operator 
depends on $n$ under this setting, 
we need to replace Assumption 
{\bf (Q1)} by the following.

\vspace{2mm}
\noindent{\bf (Q2)}: The set of real numbers $\mathbf{q}_n=\{q_{ij}^{(n)}\}_{i, j=1}^{d_n}$, $n \in \N$, satisfies the following.
\begin{itemize}
\item 
It holds that $\dis q_{ij}^{(n)}>0, \, i \neq j$ and 
\[
\dis q_{ii}^{(n)}=-\sum_{j \neq i}q_{ij}^{(n)}, 
\qquad 
i=1, 2, \dots, d_n.
\]
\item
There exists some sequence $\{a_n\}_{n=1}^\infty$ of 
positive integers satisfying $a_n=o(n^{-11/16})$ as $n \to \infty$ and
\[
\max_{i, j=1, 2, \dots, d_n}q_{ij}^{(n)} \le Ca_n, \qquad n \in \N,
\]
for some $C>0$ independent of $n \in \N$.
\end{itemize}

\begin{re}\normalfont
The approximating property 
\[
\lim_{n \to \infty}\|B_{d_n, n}^{(\mathbf{q}_n)}f-f\|_\infty=0, \qquad f \in C(\Delta_{d_n-1}),
\]
still holds even though we assume {\bf (Q2)} instead of {\bf (Q1)}, 
which is easily checked with a slight modification of the proof of 
Proposition \ref{prop:approximation-d-modified}. 
\end{re}

For $n \in \N$, let $\mathcal{Q}_n$ be 
the $d_n \times d_n$-matrix defined by
\[
\mathcal{Q}_n=Q_n-I, \qquad 
Q_n=\begin{bmatrix}
1+q_{11}^{(n)} & q_{12}^{(n)} & \cdots & q_{1d_n}^{(n)} \\
q_{21}^{(n)} & 1+q_{22}^{(n)} & \cdots & q_{2d_n}^{(n)} \\
\vdots & \vdots & \ddots & \vdots \\
q_{d_n1}^{(n)} & q_{d_n2}^{(n)} & \cdots & 1+q_{d_nd_n}^{(n)}
\end{bmatrix}. 
\]
To establish an infinite-dimensional extension of 
Theorem \ref{Thm:Bernstein-general-1}, 
we need to put 
the following additional assumption
for the mutation operator $\mathcal{Q}$
of $\mathfrak{A}$. 

\vspace{2mm}
\noindent
{\bf (Q3)}: The mutation operator 
$\mathcal{Q}$ on $C(E)$ satisfies the following.
\begin{itemize}
    \item The linear operator $\mathcal{Q}$ generates  
    a Feller semigroup $(\mathcal{T}_t)_{t \ge 0}$ on $C(E)$. 
    \item It holds that    
    \[
    \lim_{n \to \infty}\|n\mathcal{Q}_n\beta - \mathcal{Q}\beta\|_\infty =0, 
    \qquad f \in \Dom(\mathcal{Q}). 
    \]
\end{itemize}

Under Assumptions {\bf (D)}, {\bf (Q2)} and {\bf (Q3)} introduced above, 
we show that the iterate of the $d_n$-dimensional 
Bernstein operator $B_{d_n,n}^{(\mathsf{q}_n)}$ uniformly converges to
the Feller semigroup $({\sf T}^{(\infty)}_t=\e^{t\mathfrak{A}})_{t \ge 0}$ 
generated by $\mathfrak{A}$ as $n \to \infty$ together with 
its rate of convergence, which is our final main result of the present paper.

\begin{tm}
\label{Thm:Bernstein-infinity-1}
        We assume {\bf (D)}, {\bf (Q2)} and {\bf (Q3)}.
        For $ \mu \in \mathcal{P}(E) $, 
        let $\big(\mathsf{X}_t^{(\infty)}(\mu)\big)_{t \ge 0}$
        be the 
        $\mathcal{P}(E)$-valued diffusion process
        with $\mathsf{X}_0^{(\infty)}(\mu)=\mu \in \mathcal{P}(E)$
        whose infinitesimal generator is given by $\mathfrak{A}$. 
        Then, we have 
        \begin{equation*}
        {\sf T}_t^{(\infty)} \varphi(\mu)
        =\E\big[\varphi\big({\sf X}_t^{(\infty)}(\mu)\big)\big], 
        \qquad \varphi \in C(\mathcal{P}(E)), \, 
        \mu \in \mathcal{P}(E), \, t \ge 0,
        \end{equation*}
        and 
        \begin{equation}
            \lim_{n \to \infty}
            \left\|(B_{d_n,n}^{(\mathbf{q}_n)})^{\lfloor nt \rfloor}P_{d_n}\varphi 
            - P_{d_n}\left(\E\big[\varphi\big({\sf X}_t^{(\infty)}(\cdot)\big)\big]\right)\right\|_\infty=0
            , \qquad \varphi \in C(\mathcal{P}(E)), \, t \ge 0.
            \label{Eq:Bernstein-convergence-infinity}
        \end{equation}
   Moreover, we assume that 

\vspace{2mm}
\noindent
{\bf (A5):} $d_n \le Cn^{1/8-\ve}$ and 
    $a_n \le Cn^{-11/16+\ve}$ for 
    some $C>0$ and some $\ve \in (0, 1/8)$, and 

\vspace{2mm}
\noindent
{\bf (A6):} 
There exists a seminorm
$\tau_n$ on $\Dom(\mathcal{Q})$
with $\tau_n(\beta) \searrow 0$ as $n \to \infty$ and
    \[
    \|n\mathcal{Q}_n\beta - \mathcal{Q}\beta\|_\infty \le 
    \tau_n(\beta), \qquad n \in \N.
    \]

\vspace{2mm}
\noindent
   Let $\varphi(\mu)=\la \beta, \mu^{\otimes N}\ra \in \mathcal{D}$ 
   satisfy that $\mathsf{T}_t^{(\infty)}\varphi \in \mathcal{D}$ for $t \ge 0$. Namely, for each $t \ge 0$,  
   there exist $N_t \in \N$ and $\beta_t \in \Dom(\mathcal{Q}^{(N_t)}) \cap C(E^{N_t})$ such that 
   $\mathsf{T}_t^{(\infty)}\varphi(\mu)=\la \beta_t, \mu^{\otimes N_t}\ra$.     
   Then we have 
        \begin{align}
            &\left\|(B_{d_n,n}^{(\mathbf{q}_n)})^{\lfloor nt \rfloor}P_{d_n}\varphi 
            - P_{d_n}\left(\E\big[\varphi\big({\sf X}_t^{(\infty)}(\cdot)\big)\big]\right)\right\|_\infty \nn \\
            &\le 
            C\left(\sqrt{\frac{t}{n}}+\frac{t}{n}\right)
            (\tau_n(\beta) \vee n^{-\ve}) 
            +C \int_0^t (\tau_n(\beta_s) \vee n^{-\ve}) \, \dd s
            \label{Eq:Bernstein-convergence-infinity-rate}
        \end{align}
    for $n \in \N$ and $t \ge 0$, where
    the positive constant $C>0$ depends on not $n \in \N$ but $\varphi$. 
     
\end{tm}    

In the present paper, we do not discuss the functional limit theorem 
in the infinite-dimensional case. 
However, it should be a further important problem to reveal. 
So we will deal with it in the forthcoming paper.

\section{{\bf Proofs of Theorems \ref{Thm:Bernstein-1} and \ref{Thm:Bernstein-2}}}
\label{Sect:Proofs-1}

In this section, we shall give the proofs of 
Theorems \ref{Thm:Bernstein-1} and 
\ref{Thm:Bernstein-2}.

\subsection{The Voronovskaya-type theorem for $B_{d,n}$}

Before demonstrating the proofs, we need to show 
the lemma below by following the argument in 
\cite{KYZ18} basically. 
This roughly says that the sequence 
$\{n(B_{d,n} - I)\}_{n=1}^\infty$
uniformly converges to the second order differential operator $\mathcal{A}_d$
given by \eqref{Eq:generator-Bernstein}. 
Its rate of convergence is also established by using the 
expression \eqref{Eq:d-Bernstein-expectation} effectively.

\begin{lm}
\label{Lem:Voronovskaya-Berstein}
For every $f \in C^2(\Delta_{d-1})$, we have 
    \[
    \lim_{n \to \infty}
    \|n(B_{d,n}f - f) - \mathcal{A}_df\|_\infty=0,
    \]
where $\mathcal{A}_d$ is the 
second order differential operator acting on 
$C^2(\Delta_{d-1})$ given by \eqref{Eq:generator-Bernstein}. 
Moreover, if all second partial derivatives of 
$f \in C^2(\Delta_{d-1})$ are Lipschitz, we have
    \begin{equation}\label{Eq:generator-estimate}
    \|n(B_{d,n}f - f) - \mathcal{A}_df\|_\infty
        \le \left(\frac{d^{5/2}}{16 \cdot 3^{1/4}}
        \max_{i, j=1, 2, \dots, d}
        \mathrm{Lip}\left(\del_{ij}f\right) \right)
        \times \frac{1}{\sqrt{n}}, 
    \qquad n \in \N.
    \end{equation}
\end{lm}

\begin{proof}
    In order to make a notation simple, we write 
    \[
    G_n(\x)
    =\big(G_n^1(\bx), G_n^2(\bx), \dots, G_n^d(\bx)\big)
    =\left(
    \frac{1}{n}S_n^1(x_1), 
    \frac{1}{n}S_n^2(x_2),
     \dots, \frac{1}{n}S_n^d(x_d)\right)
    \]
    for $\x=(x_1, x_2, \dots, x_d) \in \Delta_{d-1}$. 
    By definition, it holds that
    \begin{equation}
    \label{Bernstein-1st-moment}
        \E\left[G_n^i(\bx)-x_i\right] = 0, 
        \qquad i = 1,2, \dots, d,
    \end{equation}
    and
    \begin{equation}
    \label{Bernstein-2nd-moment}
        \E\left[
        \left(G_n^i(\bx) - x_i\right)
        \left(G_n^j(\bx) - x_j\right)
        \right] = \frac{1}{n}x_i(\delta_{ij} - x_j), 
        \qquad i,j = 1,2, \dots, d,
    \end{equation}
    for $ \x = (x_1, x_2, \dots, x_d) \in \Delta_{d-1} $.
    By applying the Taylor formula with the integral remainder and \eqref{Bernstein-1st-moment}, 
    one has
    \begin{align*}
        &B_{d, n} f(\x) - f(\x)
        = \E\left[f\big(G_n(\x)\big) - f(\x)\right] \\
        &= \E\left[
        \sum_{i,j=1}^d 
        \int^1_0 (1-t) 
        \left(G_n^i(\bx) - x_i\right)
        \left(G_n^j(\bx) - x_j\right)
        \del_{ij}f
        \big(\x + t(G_n(\x) - \x)
        \big) \, \dd t
        \right]
    \end{align*}
    for $ \x = (x_1, x_2, \dots, x_d)\in \Delta_{d-1} $. 
    This leads to 
    \begin{align}
        n\big(B_{d, n} f(\x) - f(\x)\big) - \A_d f(\x) 
        &= n\E \Bigg[ 
        \sum^{d}_{i,j=1}
        \int^{1}_{0} 
        (1-t)
        \left(G^i_{n}(\bx) - x_{i}\right)
        \left(G^j_{n}(\bx) - x_{j}\right) \notag\\
        &\hspace{1cm}\times 
        \left\{
        \del_{ij}f
        \big(\x + t\left(G_n(\x) - \x\right)\big)
        - \del_{ij}f(\x) \right\} \,
        \dd t
        \Bigg] \notag \\
        &=: n\E\left[J_n(\x)\right]\label{J_n-A}
    \end{align}
    for $ n\in \N $ and $ \x \in \Delta_{d-1} $. 
    Since 
    $ \del_{ij}f $
    is uniformly continuous on $\Delta_{d-1}$, 
    for any $ \ve > 0 $, 
    we can choose some $ \delta > 0 $
    such that $ 0 < \|\x - \y\| < \delta $, 
    $ \x, \y \in \Delta_{d-1} $, 
    implies $ |\del_{ij}f(\x) - \del_{ij}f(\y)| < \ve$.
    Hence, it follows from the Schwarz inequality 
    and \eqref{Bernstein-2nd-moment} that
    \begin{align}
        \big|\E\left[
        J_n(\x) : \left\|G_n(\x) - \x\right\| < \delta 
        \right]\big| 
        &\le \frac{\ve}{2} \sum_{i,j=1}^d 
        \E\Big[
        \left|G^i_{n}(\bx) - x_{i}\right|
        \left|G^j_{n}(\bx) - x_{j}\right|
        \Big] \notag\\
        &\le \frac{\ve}{2} \sum_{i,j=1}^d 
        \left\{
        \E\left[\left|G^i_{n}(\bx) - x_{i}\right|^2\right]
        \E\left[\left|G^j_{n}(\bx) - x_{j}\right|^2
        \right]
        \right\}^{1/2} \notag\\
        &\le \frac{\ve}{2} \sum_{i,j=1}^d 
        \left\{
        \frac{1}{n^2}x_i(1 - x_i)x_j(1 - x_j)
        \right\}^{1/2} 
        \le \frac{\ve d^2}{8n}. \label{J_n-B}
    \end{align}
    On the other hand, the inequality
    $ \left\|G_n(\x) - \x\right\| \le 1$ implies that
    \begin{align}
        &\big|\E\left[
        J_n(\x) : \left\|G_n(\x) - \x\right\|
        \ge \delta \right]\big|\notag \\
        &\le 2 \sum_{i,j=1}^d 
        \left\|\del_{ij}f\right\|_{\infty}
        \E\left[
        \int^1_0 (1-t) 
        \left|G^i_{n}(\bx) - x_{i}\right|
        \left|G^j_{n}(\bx) - x_{j}\right| \, 
        \dd t
         : \left\|G_n(\x) - \x\right\|
        \ge \delta \right] \notag \\
        &\le \sum_{i,j=1}^d 
        \left\|\del_{ij}f\right\|_{\infty}
        \PP\big(\left\|G_n(\x) - \x\right\|
        \ge \delta \big). \notag 
    \end{align}
    Here, we have  
    \begin{align}
        \PP\big(
        \left\|G_n(\x) - \x\right\|
        \ge \delta 
        \big)
        &=
        \PP\left(
        \max_{i=1, 2, \dots, d}
        \left|
        G^i_n(\bx) - x_i
        \right|^2 \ge \frac{\delta^2}{d} \right) \notag\\
        &=
        \PP\left(
        \bigcup_{i=1}^d
        \left\{\left|
        G^i_n(\bx) - x_i
        \right|^2 \ge \frac{\delta^2}{d} \right\}\right)\notag\\
        &\le
        \sum_{i=1}^d\PP\left(
         \left|
        G^i_n(\bx) - x_i
        \right|^2 \ge \frac{\delta^2}{d}  \right)
        \le 2 d \exp\left(
        -\frac{n\delta^{4}}{2d^2}
        \right)\label{multi-Hoef}
    \end{align}
    by employing the Hoeffding inequality 
    (cf.~\cite[Theorem 1]{Hoeffding}). 
    This implies that
    \begin{equation}
    \label{J_n-C}
        \big|\E\left[
        J_n(\x) : \left\|G_n(\x) - \x \right\|
        \ge \delta \right]\big|
        \le 2 d^3 \max_{i, j=1,2,\dots,d}
        \left\|\del_{ij}f\right\|_{\infty}
        \exp\left(
        -\frac{n\delta^{4}}{2 d^2}
        \right).
    \end{equation}
    By combining \eqref{J_n-A} with 
    \eqref{J_n-B} and \eqref{J_n-C}, 
    and by letting $n \to \infty$, we obtain 
    \begin{align}
        &\left|n\big(B_{d, n} f(\x) - f(\x)\big) - \A_d f(\x)\right| \notag \\
        &\le \frac{\ve d^2}{8} 
        + 2 n d^3 \max_{i, j=1,2,\dots,d}
        \left\|\del_{ij}f\right\|_{\infty}
        \exp\left(
        -\frac{n\delta^{4}}{2d^2}
        \right)
        \to \frac{\ve d^2}{8}. 
    \end{align}
    Since $\ve > 0$ is arbitrary, 
    we conclude the desired convergence after letting $\ve \searrow 0$.
        
    Next, suppose that each
    $\del_{ij}f$, $i, j = 1,2,\dots, d$, 
    is Lipschitz continuous.
    Then, it follows from \eqref{J_n-A} that
    \begin{align}\label{estimate-final}
        &\left|n\big(B_{d, n} f(\x) - f(\x)\big) - \A_d f(\x)\right| \nn\\
        &\le n 
        \max_{i, j=1, 2, \dots, d}
        \mathrm{Lip}\left(\del_{ij}f\right)
        \E\left[
        \sum_{i, j=1}^d 
        \int^1_0 t(1-t) 
        \left|G^i_{n}(\bx) - x_{i}\right|
        \left|G^j_{n}(\bx) - x_{j}\right|
        \left\|G_n(\x) - \x\right\| \, 
        \dd t
        \right]\nn\\
        &= \frac{n}{6}
        \max_{i, j=1, 2, \dots, d}
        \mathrm{Lip}\left(\del_{ij}f\right)
        \E\left[
        \left(
        \sum_{i=1}^d 
        \left|G^i_{n}(\bx) - x_{i}\right|
        \right)^2
        \left\|G_n(\x) - \x\right\|
        \right]\nn\\
        &\le \frac{n d}{6}
        \max_{i, j=1, 2, \dots, d} 
        \mathrm{Lip}\left(\del_{ij}f\right)
        \E\left[
        \left\|G_n(\x) - \x\right\|^3
        \right] \nn\\
        &\le \frac{n d}{6}
        \max_{i, j=1, 2, \dots, d} 
        \mathrm{Lip}\left(\del_{ij}f\right)
        \E\left[
        \left\|G_n(\x) - \x\right\|^4 
        \right]^{3/4}, \qquad n \in \N, \, \x \in \Delta_{d-1},
    \end{align}
    where we apply the Jensen inequality 
    for the final line. 
    Then, we use the moment estimate 
    \begin{equation}\label{Eq;:binomial-4th}
    \E\left[\left(S_{n}^i(x_i) - n x_i\right)^{4}\right]
    =nx_i(1-x_i)(1-6x_i+6x_i^2+3nx_i-3nx_i^2)
    \le \frac{3n^2}{16}
    \end{equation}
    to deduce that
    \begin{align*}
        \E\left[
        \left\|G_n(\x) - \x\right\|^4\right]
        &=\frac{1}{n^4}\E\left[\left\|S_n(\x) - n \x \right\|^{4}\right]\\
        &\le \frac{1}{n^4}\E\left[\left(\sum^{d}_{i=1} \left(S_{n}^i(x_i) - n x_i\right)^{2}\right)^{2}\right]\\
        &\le \frac{d}{n^4} \sum^{d}_{i=1} \E\left[\left(S_{n}^i(x_i) - n x_i\right)^{4}\right]
        \le \frac{3 d^2}{16n^2}. 
    \end{align*}
     Thus, \eqref{estimate-final} is going to be
    \begin{align*}
        \left|n\left(B_{d, n} f(\x) - f(\x)\right) - \A_d f(\x)\right|
        \le \left(\frac{d^{5/2}}{16 \cdot 3^{1/4}}
        \max_{i, j=1, 2, \dots, d}
        \mathrm{Lip}\left(\del_{ij}f\right) \right)
        \times \frac{1}{\sqrt{n}}
    \end{align*}
    for $n \in \N$ and $\x \in \Delta_{d-1}$,
    which is the very desired estimate \eqref{Eq:generator-estimate}.
\end{proof}

\subsection{Proof of Theorem \ref{Thm:Bernstein-1}}
We are going to show Theorem \ref{Thm:Bernstein-1}
by making use of Trotter's approximation theorem.  
See Proposition \ref{Prop:Trotter} for more details.

\begin{proof}[Proof of Theorem {\rm \ref{Thm:Bernstein-1}}]
    We split the proof into three parts. 
    
    \vspace{2mm}
    \noindent
    {\bf Step 1.} 
    Since $C^2(\Delta_{d-1}) \subset \Dom(\mathcal{A}_d)$ and $C^2(\Delta_{d-1})$ is dense 
    in $C(\Delta_{d-1})$, we know that $\mathcal{A}_d$ is densely defined. 
    By virtue of Lemma \ref{Lem:Altomare-Campiti}
    and the Lumer–Phillips theorem (cf.~\cite[Theorem 3.1]{LF}), it turns out that 
    the closure of the differential operator $\mathcal{A}_d$ is dissipative. 
    This implies that the operator $I-\ol{\mathcal{A}}_d$ is invertible. 
    Since $C^2(\Delta_{d-1})$ is a core for $\overline{\mathcal{A}}_d$ 
    by Lemma \ref{Lem:Altomare-Campiti},
    $(I-\overline{\mathcal{A}}_d)(C^2(\Delta_{d-1}))$ is dense in $C(\Delta_{d-1})$. 
    Then, we apply Trotter's approximation theorem to conclude that 
        \begin{equation}\label{Eq:Bernstein-semigroup-conv1}
            \lim_{n \to \infty}
            \left\|(B_{d,n})^{\lfloor nt \rfloor}f 
            - {\sf T}_tf\right\|_\infty=0,
            \qquad f \in C(\Delta_{d-1}), \,\, t \ge 0. 
        \end{equation}

    \vspace{2mm}
    \noindent
    {\bf Step 2.} We here show that the contraction $C_0$-semigroup 
    $({\sf T}_t)_{t \ge 0}$ coincides with the diffusion semigroup 
    corresponding to the stochastic differential equation \eqref{SDE:d-Bernstein}. 
    For $ \x=(x_1, x_2, \dots, x_d) \in \Delta_{d-1} $ and 
    $\by=(y_1, y_2, \dots, y_d) \in \mathcal{I}$,
    we have 
    \begin{align}
        &n\E\left[
        (H_n^{N+1}(\x) - H_n^N(\x))
        (H_n^{N+1}(\x) - H_n^N(\x))^{\sf T} \, \middle| \, H_n^N(\x) = \by
        \right] \nonumber\\
        &=n\E\left[ (G_n(\by)-\by)(G_n(\by)-\by)^{{\sf T}} \right]
        = \Big(y_i(\delta_{ij} - y_j)\Big)_{i, j= 1}^d
        \label{DifApp-A}
        \end{align}
    and
        \begin{align}
        &n\E\left[H_n^{N+1}(\x) - H_n^N(\x) \, 
        \middle| \, H_n^N(\x) = \y\right] 
        =n\E\left[ G_n(\by)-\by \right]
        = \bm{0}.
        \label{DifApp-B}
    \end{align}
    Furthermore, we have  
    \begin{equation}
    \label{DifApp-C}
        \PP\left(\left\|
        H_n^{N+1}(\x) - H_n^N(\x)
        \right\| > \ve \, \middle| \, H_n^N(\x) = \y\right)
        \le
        2 d \exp\left(\frac{n \ve^4}{2d^2}\right)
    \end{equation}
    for $ \ve > 0 $ in view of \eqref{multi-Hoef}. 
    Hence, it follows from \eqref{DifApp-A}, \eqref{DifApp-B} and 
    \eqref{DifApp-C} that $ (H_n^{\lfloor nt \rfloor})_{t\ge 0}$, $n=1, 2, 3, \dots$, 
    converges weakly to the diffusion process $ (\mathsf{X}_t)_{t\ge 0} $
    which solves \eqref{SDE:d-Bernstein} as $ n \to \infty$. 
    Here, we applied the convergence criteria for sequences of 
    Markov chains given in \cite[Lemma 11.2.3]{SV}.
    In other words, we have obtained 
    \begin{equation}\label{Eq:Bernstein-semigroup-conv2}
    \lim_{n \to \infty}
            \left\|(B_{d,n})^{\lfloor nt \rfloor}f - \E\left[f({\sf X}_t(\cdot)\right)]\right\|_\infty=0,
            \qquad f \in C(\Delta_{d-1}), \,\, t \ge 0.
    \end{equation}
    Therefore, we conclude that the $C_0$-semigroup $({\sf T}_t)_{t \ge 0}$
    coincides with the diffusion semigroup of \eqref{SDE:d-Bernstein}
    once we combine \eqref{Eq:Bernstein-semigroup-conv1}
    with \eqref{Eq:Bernstein-semigroup-conv2}. 
    
    \vspace{2mm}
    \noindent
    {\bf Step 3.}  
    In the sequel, we assume both {\bf (A1)} and {\bf (A2)}. 
    We put
    \[
    \psi_n(f) := \left(\frac{d^{5/2}}{16 \cdot 3^{1/4}}
    \max_{i, j=1, 2, \dots, d}
    \mathrm{Lip}\left(\del_{ij}f\right) \right)
    \times \frac{1}{\sqrt{n}}, \quad 
    \varphi_n(f) := \psi_n(f) + \|\A_d f\|_{\infty}
    \]
    for $n \in \N$.
    Then, it holds that
    \[
    \|n(B_{d,n}f - f)\|_{\infty} \le \varphi_n(f),\qquad
    \|n(B_{d,n}f - f) - \A f\|_{\infty} \le \psi_{n}(f),
    \qquad n \in \N,
    \]
    and $ \psi_n(f) \to 0$ as $ n\to \infty $.
    Then, Proposition 
    \ref{Prop:Trotter-rate} immediately yields the
    desired rate of convergence \eqref{Eq:Bernstein-rate}.
\end{proof}

We should note that the rate of convergence of 
limit theorems for the 
iterate of $B_{d,n}$
including the Voronovskaya-type estimates 
have been discussed in various settings. 
We refer to e.g., \cite{MR}
for the case of $f \in C^3(\Delta_{d-1})$, and 
\cite{CT} for the case of 
$f \in C^2(\Delta_{d-1})$ satisfying that each $\del_{ij}f$
is $\alpha$-H\"older continuous with $\alpha \in (0, 1)$. 
In view of these studies, 
Lemma \ref{Lem:Voronovskaya-Berstein} and 
Theorem \ref{Thm:Bernstein-1} do not seem
to be new so much. However, our proof is heavily based on 
the probabilistic expression \eqref{Eq:d-Bernstein-expectation} 
of the multidimensional 
Bernstein operator, which reveals a probabilistic
interpretation \eqref{Eq:Bernstein-coinidence} of some limiting phenomena
behind the iterates of the $d$-dimensional Bernstein operator. 

It is also worth mentioning the study of 
complete asymptotic expansion of the $d$-dimensional 
Bernstein operator acting on $\Delta_{d-1}$ in 
\cite{AI}, where all coefficients in the expansion
are written down explicitly. 
The asymptotic expansion may allow us to obtain the 
Voronovskaya-type estimate directly. However, we do not 
take this approach since it might be difficult 
to extract any interesting probabilistic information from 
the asymptotic expansion.

\subsection{Proof of Theorem \ref{Thm:Bernstein-2}}
We are ready for the proof of Theorem \ref{Thm:Bernstein-2}.

\begin{proof}[Proof of Theorem {\rm \ref{Thm:Bernstein-2}}]

    Fix $\x=(x_1, x_2, \dots, x_d) \in \Delta_{d-1}$. 
    In order to show Theorem \ref{Thm:Bernstein-2},
    it is sufficient to show 
    the following two claims 
    (see e.g., \cite[Theorem 4.15]{KS}). 
    
    \begin{itemize}
    \item[{\bf (P1):}] The finite-dimensional distribution of
    $\mathcal{H}_\cdot^{(n)}(\x)$ converges to that of 
    ${\sf X}_\cdot(\x)$. Namely, it holds that 
    \begin{equation*}
        \left(\mathcal{H}_{t_1}^{(n)}(\x), \mathcal{H}_{t_2}^{(n)}(\x), \dots
        \mathcal{H}_{t_N}^{(n)}(\x)\right) 
        \to 
        \left(\mathsf{X}_{t_1}(\x), \mathsf{X}_{t_2}(\x) \dots
        \mathsf{X}_{t_N}(\x)\right) \quad \text{in law}
    \end{equation*}
    for $ N \in \N $ and 
    $ 0 \le t_1 < t_2 < \dots < t_N $. 
    
    \vspace{2mm}
    \item[{\bf (P2):}] 
    The sequence $ \{\PP \circ (\mathcal{H}^{(n)}_{\cdot})^{-1}\}_{n=1}^{\infty} $
    is tight in $C_{\bx}^{\alpha\text{{\rm -H\"ol}}}([0, \infty), \R^d)$ for all $\alpha<1/2$.
    \end{itemize}
    However, we easily see that 
    Theorem \ref{Thm:Bernstein-1} implies that 
    {\bf (P1)} is true (see e.g., \cite[Theorem 17.25]{Kallenberg}). 
    Therefore, we only to concentrate on the proof of 
    {\bf (P2)}. Our goal is then to show that 
    there exists some positive constant $ C > 0 $ independent of $n$ such that
    \begin{equation}
    \label{tight}
        \E\left[
        \left\|
        \mathcal{H}^{(n)}_{t}(\x) - \mathcal{H}^{(n)}_{s}(\x)
        \right\|^{2\beta}\right]
        \le C(t-s)^{\beta}, \qquad
        n \in \N, \,\, 0 \le s \le t, \, \beta \in \N.
    \end{equation}
    Once the moment estimate \eqref{tight} is established, the celebrated Kolmogorov continuity criterion implies that the sequence $ \{\PP \circ (\mathcal{H}^{(n)}_{\cdot})^{-1}\}_{n=1}^{\infty} $
    is tight in $C_{\bx}^{\alpha\text{{\rm -H\"ol}}}([0, \infty), \R^d)$ for all $\alpha<(\beta-1)/2\beta$. 
    Since $\beta$ can be chosen arbitrarily, 
    we conclude {\bf (P2)}. 
    
    \vspace{2mm}
    \noindent
    {\bf Step 1.}
    Let $\beta \in \N$. 
    At first, we show that 
    \begin{equation}
    \label{tight-1}
        \E\left[\left\|
        \mathcal{H}^{(n)}_{\ell/n}(\x) - \mathcal{H}^{(n)}_{k/n}(\x)
        \right\|^{2\beta}\right]
        \le
        C\left(\frac{\ell-k}{n}\right)^\beta,
        \qquad n \in \N, \,\, k, \ell \in \N_0,\,\, k \le \ell.
    \end{equation}
    for some $C > 0$ independent of $n$. 
    Since the Markov chain 
    $ \{H^N_n(\x)\}_{N=1}^\infty $ is a martingale by
    Lemma \ref{prop:martingale}, 
    we can use the Burkholder--Davis--Gundy inequality to get
    \begin{align}
        \E\left[\left\|
        \mathcal{H}^{(n)}_{\ell/n}(\x) - \mathcal{H}^{(n)}_{k/n}(\x)
        \right\|^{2\beta}\right]
        &=
        \E\left[\left\|H^\ell_n(\x) - H^k_n(\x)\right\|^{2\beta}\right]
        \notag\\
        &\le
        C_{2\beta}\,
        \E\left[ \left(\sum_{j=k}^{\ell-1}
        \left\|H_n^{j+1}(\x) - H_n^{j}(\x)\right\|^2
        \right)^{\beta}\right]\label{tight-1A},
    \end{align}
    where $ C_{2\beta}$ stands for the positive constant which appears in the upper bound for the Burkholder--Davis--Gundy inequality with the 
    exponent $ 2\beta $.  
    Then, it follows from the Markov property of 
    $ \{H^N_n(\x)\}_{N=1}^\infty $ that 
    \begin{align}
        \E\left[ \left(\sum_{j=k}^{\ell-1}
        \left\|H_n^{j+1}(\x) - H_n^{j}(\x)\right\|^2
        \right)^{\beta}\right]
        &= (\ell - k)^\beta \, 
        \E\left[
        \left\|G_n(\x) - \x\right\|^{2\beta}
        \right]\notag\\
        &\le d^{\beta-1}(\ell-k)^\beta \,
        \sum_{i=1}^d
        \E\left[
        \left|G_n^i(\bx) - x_i \right|^{2\beta}
        \right].\label{tight-1B}
    \end{align}
    By virtue of Theorem \ref{Thm:binomial-moment}
    in Appendix \ref{App:moments}, 
    we can find a constant $ C > 0 $ independent of 
    $n \in \N$ such that 
    \[
    \E\big[(S_n(x_i) - n x_i)^{2\beta}\big] \le C n^\beta, \qquad i=1, 2, \dots, d.  
    \]
    Hence, it holds that 
    \begin{equation}
    \label{tight-1C}
        \E\left[(G_n^i(\bx) - x_i)^{2\beta}\right] \le \frac{C}{n^\beta}.
    \end{equation}
    By combining \eqref{tight-1A} with 
    \eqref{tight-1B} and 
    \eqref{tight-1C}, we reach the moment estimate \eqref{tight-1}. 
    
    \vspace{2mm}
    \noindent
    {\bf Step 2.}
    Next, we aim to show \eqref{tight}.
    Let $0 \le s \le t$. 
    We take $ 1 \le k \le \ell $ satisfying 
     $ k/n \le s < (k+1)/n $ and $ \ell/n \le s < (\ell+1)/n $.
    Since the stochastic process $ (\mathcal{H}_{t}^{(n)}(\x))_{t \ge 0} $ is defined through the 
    linear interpolation, we have that 
    \begin{align*}
        \left\|\mathcal{H}^{(n)}_{(k+1)/n}(\x) - \mathcal{H}^{(n)}_{s}(\x)\right\|
        &=
        (k - n s)
        \left\|\mathcal{H}^{(n)}_{(k+1)/n}(\x) - \mathcal{H}^{(n)}_{k/n}(\x)\right\|,\\
        \left\|\mathcal{H}^{(n)}_{t}(\x) - \mathcal{H}^{(n)}_{\ell/n}(\x)\right\|
        &=
        (n t - \ell)
        \left\|\mathcal{H}^{(n)}_{(\ell+1)/n}(\x) - \mathcal{H}^{(n)}_{\ell/n}(\x)\right\|.
    \end{align*}
    By using \eqref{tight-1} and the triangle inequality, we obtain
    \begin{align*}
        &\E\left[
        \left\|
        \mathcal{H}^{(n)}_{t}(\x) - \mathcal{H}^{(n)}_{s}(\x)
        \right\|^{2\beta}\right]\\
        &\le 
        3^{2\beta-1}\left\{
        (n t - l)^{2\beta} \times C\left(\frac{1}{n}\right)^{\beta}
        + C\left(\frac{l - k - 1}{n}\right)^{\beta}
        + (k + 1 - n s)^{2\beta} \times C\left(\frac{1}{n}\right)^{\beta}
        \right\}\\
        &\le 
        C\left\{
        \left(t - \frac{l}{n}\right)^\beta 
        + \left(\frac{l}{n} - \frac{k+1}{n}\right)^\beta + 
        \left(\frac{k+1}{n} - s\right)^\beta
        \right\}
        \le C(t - s)^{\beta}, \qquad n \in \N,
    \end{align*}
    which is what we want to show.
\end{proof}

It should be noted that there are several
functional limit theorems for the 
(possibly multidimensional) Wright--Fisher models.
Guess  showed in \cite{Guess} that the sequence of 
properly scaled
stochastic process induced  by the one-dimensional
Wright--Fisher model converges weakly to
a diffusion process in $D([0, \infty); \, \R)$, 
the space of all c\`adl\`ag functions on $[0, \infty)$
equipped with the Skorokhod $J_1$-topology. 
On the other hand, Sato treated
highly generalized situations related to 
multidimensional Wright--Fisher models and 
obtained the functional limit theorem 
for the models in $C([0, \infty); \, \R^d)$ in \cite{Sato} . 
We emphasize that Sato found a solution to a 
martingale problem corresponding to the Wright--Fisher model
to deduce the desired convergence. 
In our case, we have showed the convergence of the discrete
semigroup generated by the Wright--Fisher model to the 
diffusion semigroup $({\sf T}_t)_{t \ge 0}$. 
This readily implies the convergence of 
finite-dimensional distribution of the Wright--Fisher diffusion, 
which gives one of sufficient condition to obtain the functional
limit theorem. In this sense, we provide another kind of 
approach to such a convergence from both probabilistic and 
functional-analytic perspectives.

\section{{\bf Proofs of Theorems \ref{Thm:Bernstein-general-1} and 
\ref{Thm:Bernstein-general-2}}}
\label{Sect:Proofs-2}

\subsection{The Voronovskaya-type theorem for $B_{d,n}^{(\mathbf{q}_n)}$}

Throughout this section, we always assume {\bf (Q1)}. 
For $n \in \mathbb{N}, i=1, 2, \dots, d$ and 
$\bx=(x_1, x_2, \dots, x_d) \in \Delta_{d-1}$, we put
    \[
    h_n^i(\bx):=x_i^{(\mathbf{q}_n)} - x_i 
    = \sum_{j=1}^d q_{ji}^{(n)}x_j.
    \]
Then, we note that 
    \begin{align}
    |h_n^i(\x)| &\le \sum_{j=1}^d \left|\frac{1}{n}q_{ji}+\frac{C}{n^\gamma}\right|
    \le \frac{C}{n}, 
    \qquad i=1, 2, \dots, d, \, 
    n \in \N, \, \x \in \Delta_{d-1}, 
    \label{h-estimate} 
    \end{align}
for some $C>0$ independent of $n \in \N$. 

The following is the Voronovskaya-type theorem
for the Bernstein operator associated with 
$\mathbf{q}_n$ satisfying {\bf (Q1)}.

\begin{lm}
\label{Lem:Voronovskaya-Berstein-genreal}
Suppose that $\mathbf{q}_n=\{q_{ij}^{(n)}\}_{i,j=1}^d$ satisfies {\bf (Q1)}. 
Then, for every $f \in C^2(\Delta_{d-1})$, we have 
    \[
    \lim_{n \to \infty}
    \|n(B_{d,n}^{(\mathbf{q}_n)}f - f) 
    - \mathcal{A}_d^{(\mathbf{q})}f\|_\infty=0,
    \]
where $\mathcal{A}_d^{(\mathbf{q})}$ is the 
second order differential operator acting on 
$C^2(\Delta_{d-1})$ given by \eqref{Eq:generator-Bernstein-general}. 
Moreover, if all second partial derivatives of 
$f \in C^2(\Delta_{d-1})$ are Lipschitz, we have
    \begin{equation}\label{Eq:generator-estimate-general}
    \|n(B_{d,n}^{(\mathbf{q}_n)}f - f) - \mathcal{A}_d^{(\mathbf{q})}f\|_\infty
        \le  \frac{C}{\sqrt{n}}, 
    \qquad n \in \N,
    \end{equation}
    for some $C>0$ depending on 
    $d$, $\|\del_i f\|_\infty$, 
    $\|\del_{ij}f\|_\infty$ and 
    $\mathrm{Lip}(\del_{ij}f)$ for 
    $i, j=1, 2, \dots, d$. 
\end{lm}

\begin{proof}
    Let us put
    \[
    \begin{aligned}
    G_n^{(\mathbf{q}_n)}(\x) 
    &=\Big((G_n^{(\mathbf{q}_n)})^1(\bx), (G_n^{(\mathbf{q}_n)})^2(\bx), 
    \dots, (G_n^{(\mathbf{q}_n)})^d(\bx)\Big) \\
    &=\left(
    \frac{1}{n}S_n^1(x_1^{(\mathbf{q}_n)}), 
    \frac{1}{n}S_n^2(x_2^{(\mathbf{q}_n)}),
     \dots, \frac{1}{n}S_n^d(x_d^{(\mathbf{q}_n)})\right)
    \end{aligned}
    \]
    for $\x=(x_1, x_2, \dots, x_d) \in \Delta_{d-1}$. 
    Then, we have
    \begin{equation}
    \label{Bernstein-1st-moment-general}
        \E\left[(G_n^{(\mathbf{q}_n)})^i(\bx)-x_i\right] = 
        \sum_{j=1}^d q_{ji}^{(n)}x_j=
        h_n^i (\bx)
    \end{equation}
    and
    \begin{align}
    \label{Bernstein-2nd-moment-general}
        &\E\left[
        \left((G_n^{(\mathbf{q}_n)})^i(\bx) - x_i\right)
        \left((G_n^{(\mathbf{q}_n)})^j(\bx) - x_j\right)
        \right] \notag \\
        &= \frac{1}{n}x_i^{(\mathbf{q}_n)}(\delta_{ij} - x_j^{(\mathbf{q}_n)})
        +h_n^i(\bx) h_n^j(\bx) \notag \\  
        &= \frac{1}{n}x_i(\delta_{ij} - x_j) 
        +\frac{1}{n}h_n^i(\bx)(\delta_{ij}-x_j)
        - \frac{1}{n}x_ih_n^j(\bx)
        +\left(1+\frac{1}{n}\right)
        h_n^i(\bx) h_n^j(\bx)
    \end{align}
    for $ i,j = 1,2, \dots, d $ and 
    $ \x = (x_1, x_2, \dots, x_d) \in \Delta_{d-1} $.
    It then follows from 
    \eqref{Bernstein-1st-moment-general}, 
    \eqref{Bernstein-2nd-moment-general}
    and Taylor's formula that
    \begin{align}
        &n\big(B_{d, n}^{(\mathbf{q}_n)} f(\x) - f(\x)\big) - \A_d^{(\mathbf{q})} f(\x) \notag \\
        &= \sum_{i=1}^d \left(\sum_{j=1}^d(nq_{ji}^{(n)} -  q_{ji})x_j\right)
        \del_i f(\bx)  \notag \\
        &\hspace{1cm}+\frac{1}{2}\sum_{i,j=1}^d 
        \Big(h_n^i(\bx)(\delta_{ij}-x_j)-x_ih_n^j(\bx)
        +(n+1)h_n^i(\bx)h_n^j(\bx)\Big)\del_{ij}f(\bx) \notag \\
        &\hspace{1cm}+n\E \Bigg[ 
        \sum^{d}_{i,j=1}
        \int^{1}_{0} 
        (1-t)
        \left((G_n^{(\mathbf{q}_n)})(\bx) - x_{i}\right)
        \left((G_n^{(\mathbf{q}_n)})(\bx) - x_{j}\right) \notag\\
        &\hspace{1cm}\times 
        \left(
        \del_{ij}f
        \left(\x + t\left(G_n^{(\mathbf{q}_n)}(\x) - \x\right)\right)
        - \del_{ij}f(\x) \right) \,
        \dd t
        \Bigg] \notag \\
        &=: I_n^{(1)}(\bx) + I_n^{(2)}(\bx) + n\mathbb{E}[J_n(\bx)] , \qquad \bx \in \Delta_{d-1}.  
        \label{J_n-A-general}
    \end{align}
        By using {\bf (Q1)}, we easily have 
        \begin{align}\label{Est-I_1}
           |I_n^{(1)}(\x)|
           &\le \sum_{i=1}^d \left(\sum_{j=1}^d|nq_{ji}^{(n)} -  q_{ji}|\right)
        \|\del_i f\|_\infty 
        \le \frac{C}{n^\gamma},
        \qquad n \in \N, \, \x \in \Delta_{d-1},
        \end{align}
        where the constant $C>0$ depends on  
        $d$ and $\|\del_i f\|_\infty$, $i=1, 2, \dots, d$. 
        Moreover, it follows from \eqref{h-estimate}
        that 
        \begin{align}\label{Est-I_2}
           |I_n^{(2)}(\x)| 
           &\le \frac{1}{2}\sum_{i,j=1}^d \Big(
           2|h_n^i(\x)|+|h_n^j(\x)|+(n+1)|h_n^i(\x)||h_n^j(\x)|
           \Big)
        \|\del_{ij} f\|_\infty \le \frac{C}{n}
        \end{align}
        for $n \in \N$ and $\x \in \Delta_{d-1}$. 
        Note that the constant $C>0$ in the right-hand side of \eqref{Est-I_2}
        depends on $d$ and $\|\del_{ij} f\|_\infty$, $i, j=1, 2, \dots, d$. 

    The rest is to show that $n\mathbb{E}[J_n(\bx)]$
    converges to $0$ as $n \to \infty$ uniformly in $\bx \in \Delta_{d-1}$. 
    Due to the uniform continuity of  
    $ \del_{ij}f $ on $\Delta_{d-1}$, 
    for any $ \ve > 0 $, 
    there is some $ \delta > 0 $
    such that $ 0 < \|\x - \y\| < \delta $, 
    $ \x, \y \in \Delta_{d-1} $, 
    implies $ |\del_{ij}f(\x) - \del_{ij}f(\y)| < \ve$.
    By using \eqref{h-estimate},  it holds that
    \begin{align}
        &n\left|\E\left[
        J_n(\x) : 
        \left\|G_n^{(\mathbf{q}_n)}(\x) - \x\right\| < \delta 
        \right]\right| \notag \\
        &\le \frac{\ve n}{2} \sum_{i,j=1}^d 
        \left\{
        \E\left[\left|(G_n^{(\mathbf{q}_n)})^i(\bx) - x_{i}\right|^2\right]
        \E\left[\left|(G_n^{(\mathbf{q}_n)})^j(\bx) - x_{j}\right|^2
        \right]
        \right\}^{1/2} \notag\\
        &\le \frac{\ve}{2} \sum_{i,j=1}^d 
        \left(
        \left\{x_i^{(\mathbf{q}_n)}(1 - x_i^{(\mathbf{q}_n)})
        +nh_n^i(\bx)\right\}
        \left\{x_j^{(\mathbf{q}_n)}(1 - x_j^{(\mathbf{q}_n)})
        +nh_n^j(\bx)\right\}
        \right)^{1/2} \notag \\
        &\le \frac{\ve d^2}{2} 
        \left( \frac{1}{4}+C\right), \qquad \bx \in \Delta_{d-1}. 
        \label{J_n-B-general}
    \end{align}
    Furthermore, it holds that
    \begin{align}
        &\PP\left(
        \left\|G_n^{(\mathbf{q}_n)}(\x) - \x\right\|
        \ge \delta 
        \right) \notag \\
        &\le
        \sum_{i=1}^d\PP\left(
        \left|
        (G_n^{(\mathbf{q}_n)})^i(\bx) - x_i
        \right|^2 \ge \frac{\delta^2}{d} \right)\notag\\
        &\le \sum_{i=1}^d\PP\left(
        \left|
        (G_n^{(\mathbf{q}_n)})^i(\bx) - x_i^{(\mathbf{q}_n)}
        \right|^2 \ge \frac{\delta^2}{2d}-\frac{C}{n^2}>0 \right)\notag\\
        &\le 2 d \exp\left(
        -\frac{n}{2}\left( \frac{\delta^2}{2d}-\frac{C}{n^2}\right)
        \right), \qquad \x \in \Delta_{d-1},
        \label{J_n-C-Hoeff}
    \end{align}
    for sufficiently large $n$, 
    by applying \eqref{h-estimate} and 
    the Hoeffding inequality. 
    Thus, we have
    \begin{align}
        &n\left|\E\left[
        J_n(\x) : \left\|G_n^{(\mathbf{q}_n)}(\x) - \x\right\|
        \ge \delta \right]\right| \notag \\
        &\le n\sum_{i,j=1}^d 
        \left\|\del_{ij}f\right\|_{\infty}
        \PP\left(\left\|G_n^{(\mathbf{q}_n)}(\x) - \x\right\|
        \ge \delta \right) \notag \\
        &\le 2d^3n\max_{i, j=1,2,\dots,d}
        \left\|\del_{ij}f\right\|_{\infty}
        \exp\left(
        -\frac{n}{2}\left( \frac{\delta^2}{2d}-\frac{C}{n^2}\right) 
        \right), \qquad  \bx \in \Delta_{d-1}, 
        \label{J_n-C-general}
    \end{align}
    for sufficiently large $n$. 
    By combining \eqref{J_n-B-general} with 
    \eqref{J_n-C-general}, we obtain that 
    $\|n\mathbb{E}[J_n(\cdot)]\|_\infty \to 0$ 
    as $n \to \infty$. 
    This uniform convergence together with 
    \eqref{Est-I_1} and \eqref{Est-I_2} 
    readily implies the desired convergence.

    Next, we assume that each
    $\del_{ij}f$, $i, j = 1,2,\dots, d$, 
    is Lipschitz continuous.
    Then, \eqref{J_n-A-general}, \eqref{Est-I_1} and \eqref{Est-I_2} lead to
    \begin{align}\label{Lip-est-general}
        &\left|n\big(B_{d, n}^{(\mathbf{q}_n)} f(\x) - f(\x)\big) - \A_d^{(\mathbf{q})} f(\x)\right| \nn\\
        &\le \frac{C}{n^\gamma}+\frac{C}{n}+n 
        \max_{i, j=1, 2, \dots, d}
        \mathrm{Lip}\left(\del_{ij}f\right)
        \E\Bigg[
        \sum_{i, j=1}^d 
        \int^1_0 t(1-t) \nn \\
        &\hspace{1cm}\times 
        \left|(G^{(\mathbf{q}_n)}_{n})^i(\bx) - x_{i}\right|
        \left|(G^{(\mathbf{q}_n)}_{n})^j(\bx) - x_{j}\right|
        \left\|G_n^{(\mathbf{q}_n)}(\x) - \x\right\| \, 
        \dd t
        \Bigg]\nn\\
        &\le\frac{C}{n}+ Cn \times 
        \E\left[
        \left\|G_n^{(\mathbf{q}_n)}(\x) - \x\right\|^4 
        \right]^{3/4}, \qquad n \in \N, \, \x \in \Delta_{d-1}.
    \end{align}
    Here, \eqref{Eq;:binomial-4th} and {\bf (Q1)}
    implies that
    \begin{align*}
        &\E\left[
        \left\|G_n^{(\mathbf{q}_n)}(\x) - \x\right\|^4\right] \\
        &\le \frac{d}{n^4} \sum^{d}_{i=1} \E\left[\left(S_{n}^i(x_i^{(\mathbf{q}_n)}) - n x_i\right)^{4}\right] \\
        &\le \frac{8d}{n^4}\left\{
        \E\left[ \left(S_{n}^i(x_i^{(\mathbf{q}_n)}) - n x_i^{(\mathbf{q}_n)}\right)^{4}\right]
        +(nx_i^{(\mathbf{q}_n)}-nx_i)^4\right\}\\
        &\le \frac{8d}{n^4} 
        \left(\frac{3n^2}{16}+C\right)
        \le \frac{C}{n^2}, \qquad \x \in \Delta_{d-1}. 
    \end{align*}
    Therefore, \eqref{Lip-est-general} 
    is going to be
    \begin{align*}
        \left|n\big(B_{d, n}^{(\mathbf{q}_n)} f(\x) - f(\x)\big) - \A_d^{(\mathbf{q})} f(\x)\right|
        \le  \frac{C}{\sqrt{n}}
    \end{align*}
    for $n \in \N$ and $\x \in \Delta_{d-1}$,
    where the constant $C>0$ depends on 
    $d$, $\|\del_i f\|_\infty$, 
    $\|\del_{ij}f\|_\infty$ and 
    $\mathrm{Lip}(\del_{ij}f)$ for 
    $i, j=1, 2, \dots, d$. 
\end{proof}

\subsection{Proofs of Theorems \ref{Thm:Bernstein-general-1} and 
\ref{Thm:Bernstein-general-2}}

Since we have established the Voronovskaya-type theorem 
for $B_{d, n}^{(\mathbf{q}_n)}$, it allows us to 
give the proof of Theorem \ref{Thm:Bernstein-general-1}.

\begin{proof}[Proof of Theorem {\rm \ref{Thm:Bernstein-general-1}}]
    
    We follow the same argument in the proof of Theorem \ref{Thm:Bernstein-1}.
    By virtue of  
    Lemmas \ref{Lem:core-general} and \ref{Lem:Voronovskaya-Berstein-genreal}, 
    Trotter's approximation theorem implies that
    \begin{equation}\label{Eq:Bernstein-general-semigroup-conv1}
        \lim_{n \to \infty}
        \left\|(B_{d,n}^{(\mathbf{q}_n)})^{\lfloor nt \rfloor}f 
        - {\sf T}_t^{(\mathbf{q}_n)} f\right\|_\infty=0,
        \qquad f \in C(\Delta_{d-1}), \,\, t \ge 0. 
    \end{equation}
    On the other hand, it follows from \eqref{Bernstein-2nd-moment-general} 
    and {\bf (Q1)} that
    \begin{align}
        &n\E\left[
        \left((H_n^{(\mathbf{q}_n)})^{N+1}(\x) 
        - (H_n^{(\mathbf{q}_n)})^N(\x)\right)_i
        \left((H_n^{(\mathbf{q}_n)})^{N+1}(\x) 
        - (H_n^{(\mathbf{q}_n)})^N(\x)\right)_j
        \, \middle| \, (H_n^{(\mathbf{q}_n)})^N(\x) = \by
        \right] \nonumber\\
        &=n\E\left[
        ((G_n^{(\mathbf{q}_n)})^{i}(\by) - y_i)
        ((G_n^{(\mathbf{q}_n)})^j(\by) - y_j)
        \right] \nonumber \\
        &= 
        y_i(\delta_{ij} - y_j) 
        +h_n^i(\by)(\delta_{ij}-y_j)
        - y_i h_n^j(\by)
        + (n + 1) h_n^i(\by) h_n^j(\by)
        \nonumber\\
        &\to 
        y_i(\delta_{ij} - y_j), \qquad i, j = 1, 2, \dots, d,
        \label{DifApp-general-A}
        \end{align}
    and
    \begin{align}
        &n\E\left[
        \left((H_n^{(\mathbf{q}_n)})^{N+1}(\x) 
        - (H_n^{(\mathbf{q}_n)})^N(\x) \right)_i
        \, \middle| \, (H_n^{(\mathbf{q}_n)})^N(\x) = \y\right] 
        \nonumber \\
        &=n\E\left[ (G_n^{(\mathbf{q}_n)})^i(\by) - y^{(\mathbf{q}_n)}_i
        + y^{(\mathbf{q}_n)}_i - y_i \right]
        \nonumber \\
        &= n \sum_{j=1}^d q_{ji}^{(n)} y_j
        \to 
        \sum_{j=1}^d q_{ji} y_j, \qquad i = 1, 2, \dots, d,
        \label{DifApp-general-B}
    \end{align}
    as $n \to \infty$, for $ \x=(x_1, x_2, \dots, x_d) \in \Delta_{d-1} $ and 
    $\by=(y_1, y_2, \dots, y_d) \in \mathcal{I}$.
    Moreover, by using \eqref{J_n-C-Hoeff}, we have  
    \begin{align}
        &\PP\left(\left\|
        (H_n^{(\mathbf{q}_n)})^{N+1}(\x) 
        - (H_n^{(\mathbf{q}_n)})^N(\x)
        \right\| > \ve 
        \, \middle| \, (H_n^{(\mathbf{q}_n)})^N(\x) = \y\right)
        \nonumber \\
        &\le
        2 d \exp\left(
        -\frac{n}{2}\left(\frac{\ve^2}{2d}-\frac{C}{n^2}\right)
        \right)
        \to 0
        \label{DifApp-general-C}
    \end{align}
    as $n \to \infty$, for every $ \ve > 0 $. 
    Then, due to \eqref{DifApp-general-A}, \eqref{DifApp-general-B}
    and \eqref{DifApp-general-C}, we can use \cite[Lemma 11.2.3]{SV} 
    to deduce that 
    $ \big((H_n^{(\mathbf{q}_n)})^{\lfloor nt \rfloor}\big)_{t\ge 0}$, 
    $n=1, 2, 3, \dots$, 
    converges weakly to the solution to  \eqref{SDE:d-Bernstein-general} as $ n \to \infty$.
    Thus, the combination of this weak convergence with 
    \eqref{Eq:Bernstein-general-semigroup-conv1} 
    obviously concludes \eqref{Eq:Bernstein-convergence-general}. 

    Under the assumptions {\bf (A3)} and {\bf (A4)}, 
    the rate of convergence \eqref{Eq:Bernstein-rate-general}
    is also obtained 
    by applying Proposition \ref{Prop:Trotter-rate}
    and \eqref{Eq:generator-estimate-general},
    similarly to the proof of 
    \eqref{Eq:Bernstein-rate}. 
\end{proof}

In the end of this section, we give the proof of Theorem \ref{Thm:Bernstein-general-2},
in which the martingale property (Proposition \ref{prop:martingale})
also plays a crucial role.

\begin{proof}[Proof of Theorem {\rm \ref{Thm:Bernstein-general-2}}]
    For the proof,
    it is sufficient to show the tightness of the sequence
    $ \big\{\PP \circ \big((\mathcal{H}^{(\mathbf{q}_n)}_{\cdot})^{(n)}\big)^{-1}
    \big\}_{n=1}^{\infty} $
    in $C_{\bx}^{\alpha\text{{\rm -H\"ol}}}([0, \infty), \R^d)$ for all $\alpha<1/2$.
    Let $\beta \in \N$. 
    By applying Proposition \ref{prop:martingale} and
    the Burkholder--Davis--Gundy inequality, we obtain
    \begin{align*}
        \E\left[\left\|
        (\mathcal{H}^{(\mathbf{q}_n)}_{\ell/n})^{(n)}(\x) 
        - (\mathcal{H}^{(\mathbf{q}_n)}_{k/n})^{(n)}(\x)
        \right\|^{2\beta}\right] 
        \le
        C(\ell - k)^{\beta}\,
        \E\left[ \left\|
        G^{(\mathbf{q}_n)}_n(\bx) - \bx^{(\mathbf{q}_n)}
        \right\|^{2\beta}\right],
    \end{align*}
    for $ n\in \N$ and $k, \ell \in \N_0$ with $k\le \ell$, 
    where $C>0$ is some constant independent of $n \in \mathbb{N}$.
    Then, Proposition \ref{Thm:binomial-moment} yields 
    \begin{align*}
        &\E\left[\left\|
        (\mathcal{H}^{(\mathbf{q}_n)}_{\ell/n})^{(n)}(\x) 
        - (\mathcal{H}^{(\mathbf{q}_n)}_{k/n})^{(n)}(\x)
        \right\|^{2\beta}\right] \\
        &\le
        Cd^{\beta-1}(\ell - k)^{\beta}\sum_{i=1}^d
        \E\left[ |
        (G^{(\mathbf{q}_n)}_n)^i(\bx) - x_i^{(\mathbf{q}_n)}
        |^{2\beta}\right] \le C\left(\frac{\ell-k}{n}\right)^{\beta}, 
        \qquad n \in \mathbb{N}. 
    \end{align*}
    Therefore, we can deduce that 
    \[
        \E\left[\left\|
        (\mathcal{H}^{(\mathbf{q}_n)}_{t})^{(n)}(\x) 
        - (\mathcal{H}^{(\mathbf{q}_n)}_{s})^{(n)}(\x)
        \right\|^{2\beta}\right] \le C(t-s)^\beta,
        \qquad n \in \mathbb{N}, \, 0 \le s \le t,
    \]
    for some $C>0$ independent of $n \in \mathbb{N}$, 
    by following the same argument 
    as in the proof of Theorem \ref{Thm:Bernstein-2}.
    This moment estimate readily implies the tightness of 
    $ \{\PP \circ ((\mathcal{H}^{(\mathbf{q}_n)}_{\cdot})^{(n)})^{-1}\}_{n=1}^{\infty} $
    in $C_{\bx}^{\alpha\text{{\rm -H\"ol}}}([0, \infty), \R^d)$ for all $\alpha<1/2$.
\end{proof}


\section{{\bf Proof of Theorem \ref{Thm:Bernstein-infinity-1}}}
\label{Sect:Fleming--Viot}

    In this section,  
    we review our infinite-dimensional framework and mention that two major examples 
    of measure-valued diffusions arising in populatin genetics are in scope of our framework.
    Moreover, we show the 
    Voronovskaya-type theorem for $\{n(B_{d_n,n}^{(\mathbf{q}_n)})-I)\}_{n=1}^\infty$ 
    and the uniform convergence of $(B_{d_n,n}^{(\mathbf{q}_n)})^{\lfloor nt \rfloor}$ 
    to the Fleming--Viot diffusion semigroup generated by $\mathfrak{A}$.

\subsection{Examples of mutation operators}



Recall that $E$ is a compact metric space and $\mathcal{P}(E)$ the set of 
all Borel probability measures on $E$, which is also compact. 
The main objective in this section is the 
linear operator $\mathfrak{A}$, acting on $C(\mathcal{P}(E))$, 
defined by \eqref{Eq:generator-FV}. 
In order to discuss limit theorems, we need to take a discretization 
$E^{(d_n)}=\{z_1^{(d_n)}, z_2^{(d_n)}, \dots, z_{d_n}^{(d_n)}\} \subset E$
along the sequence $\{d_n\}_{n=1}^\infty$ of dimensions satisfying {\bf (D)}. 
Furthermore, we also need to assume some technical but natural assumptions 
{\bf (Q2)} and {\bf (Q3)}
for the convergence of $\mathbf{q}_n=\{q_{ij}^{(n)}\}_{i,j=1}^{d_n}$
to the mutation operator $\mathcal{Q}$ of \eqref{Eq:generator-FV}, which generates a 
Feller semigroup on $C(E)$. 

Before giving the proofs of main results, 
it is worth mentioning that some historically 
important examples in the study of population genetics does satisfy 
both {\bf (Q2)} and {\bf (Q3)}.

\begin{ex}[Ohta--Kimura model]\normalfont\label{Ex:Ohta-Kimura}
Let $E=\R \cup \{\Delta\}$ be 
the one-point compactification of $\R$
and 
\[
\mathcal{Q}f(z)=\frac{\theta}{2}f''(z),
\qquad \Dom(\mathcal{Q})=\{f \in C(E) \mid (f-f(\Delta))|_{\R} \in C_c^2(\R)\},
\]
for some $\theta>0$, where 
$C_c^2(\R)$ denotes the set of 
twice differentaible functions 
$f : \R \to \R$ such that $f, f'$ and $f''$
are compactly supported. 
We can easily verify that the linear operator 
$\mathcal{Q}$ is obtained by the limit of 
the {\it Ohta--Kimura model} (cf.~\cite{OK}), that is, 
the discretization of $E=\R \cup \{\Delta\}$ is given by 
\[
E^{(d_n)}=\begin{cases}
    \dis\frac{1}{\sqrt{d_n}}\left\{ 0, \pm 1, 
    \pm 2, \dots, \pm (d_n/2-1), d_n/2\right\} & \text{if $d_n$ is even}\\
    &\\
    \dis\frac{1}{\sqrt{d_n}}\left\{ 0, \pm 1, 
    \pm 2, \dots, \pm (d_n-1)/2\right\} & \text{if $d_n
    $ is odd}
\end{cases}
\]
for $n=1, 2, 3, \dots$, and 
the bounded linear operator $\mathcal{Q}_n$ is defined by
\[
\mathcal{Q}_n f\left(\frac{i}{\sqrt{d_n}}\right)
=\frac{\theta d_n}{2n}\left\{ f\left(\frac{i-1}{\sqrt{d_n}}\right)
+f\left(\frac{i+1}{\sqrt{d_n}}\right)
-2f\left(\frac{i}{\sqrt{d_n}}\right)\right\}
\]
for all $i$ with $i/\sqrt{d_n} \in E^{(d_n)}$,
regarding $f(j/\sqrt{d_n})=0$ unless $j/\sqrt{d_n} \in E^{(d_n)}$. 
In this case, $\mathbf{q}_n=\{q_{ij}^{(n)}\}_{i,j=1}^{d_n}$
is given by 
\[
q_{ij}^{(n)}=\begin{cases}
    \theta d_n/2n & \text{if }j=i \pm 1 \\
    -\theta d_n/n & \text{if }j=i \\
    0 & \text{otherwise}
\end{cases}, \qquad n \in \N,
\]
which satisfies both {\bf (Q2)} and {\bf (Q3)}.
\end{ex}

\begin{ex}\normalfont\label{Ex:Kimura-Crow}
Another example is the so-called 
{\it infinitely-many-neutral-alleles model with uniform mutation}
discussed in e.g., 
Kimura--Crow \cite{KC} and 
Ethier--Kurtz \cite{EK81, EK}, 
which is given by $E=[0, 1]$ and 
\[
\mathcal{Q}f(z)=\frac{\theta}{2}\int_0^1 
\big(f(y)-f(z)\big) \, \dd y
\]
for some $\theta>0$. 
Note that the adjective ``{\it neutral}''
means the lack of selections. 
Actually, this model is obtained by the limit of 
the discrete neutral-alleles model with 
uniform mutation, that is, 
the discretization of $E=[0, 1]$ is 
$E^{(d_n)}=\{i/d_n \mid i=1, 2, \dots, d_n\}$, $n=1, 2, 3, \dots$, and 
the bounded linear operator $\mathcal{Q}_n$ is given by
\[
\mathcal{Q}_nf\left(\frac{i}{d_n}\right)
=\frac{\theta}{2n(d_n-1)}
\sum_{j=1}^{d_n}
\left\{ f\left(\frac{j}{d_n}\right)
-f\left(\frac{i}{d_n}\right)\right\}, \qquad 
i=1, 2, \dots, d_n,
\]
with 
\[
q_{ij}^{(n)}=\begin{cases}
    \theta/2n(d_n-1) & \text{if }i \neq j \\
    -\theta/2n & \text{if }i=j
\end{cases}, \qquad n \in \N.
\]
Note that $\mathbf{q}_n=\{q_{ij}^{(n)}\}_{i,j=1}^{d_n}$
obviously satisfies both {\bf (Q2)} and {\bf (Q3)}.
\end{ex}

\subsection{The Voronovskaya-type theorem for $B_{d_n, n}^{(\mathbf{q}_n)}$}

We show the Voronovskaya-type theorem
for the Bernstein operator 
$B_{d_n, n}^{(\mathbf{q}_n)}$ with 
a sequence of dimension $\{d_n\}_{n=1}^\infty$ satisfying {\bf (D)}. 
Consequently, we will see that the infinitesimal generator
$\mathfrak{A}$ of the Fleming--Viot process appears in the limit.

\begin{tm}\label{Thm:Voronovskaya-infinite}
We assume {\bf (D)}, {\bf (Q2)} and {\bf (Q3)}. 
Then, we have 
    \[
    \lim_{n \to \infty}\|n(B_{d_n, n}^{(\mathbf{q}_n)}-I)P_{d_n}\varphi
     - P_{d_n}\mathfrak{A}\varphi \|_\infty=0, \qquad 
     \varphi \in \mathcal{D}.
    \]
Moreover, if we assume {\bf (A5)} and {\bf (A6)}, 
then we have 
    \begin{equation}\label{Eq:Voronovskaya-infinite-rate}
    \|n(B_{d_n, n}^{(\mathbf{q}_n)}-I)P_{d_n}\varphi
     - P_{d_n}\mathfrak{A}\varphi \|_\infty 
     \le C(\tau_n(\beta) \vee n^{-\ve}), \qquad n \in \N, \, \varphi \in \mathcal{D},
    \end{equation}
for $\varphi(\mu)=\la \beta, \mu^{\otimes N}\ra \in \mathcal{D}$,
where $C>0$ depends on not $n \in \N$ but $\varphi$. 
\end{tm}

As for Assumotions {\bf (A5)} and {\bf (A6)}, 
see the statement of Theorem \ref{Thm:Bernstein-infinity-1}. 

\begin{proof}
    Let $\varphi(\mu)=\la \beta, \mu^{\otimes N}\ra \in \mathcal{D}$, where 
    $N \in \N$ and $\beta \in \Dom(\mathcal{Q}^{(N)}) \cap C(E^N)$. 
    Then, it follows from \eqref{Eq:generator-FV-2} that
    \[
    \begin{aligned}
    P_{d_n}\mathfrak{A}\varphi(\bx)
    &= \sum_{1 \le \ell_1 < \ell_2 \le N}\Big(
    \la \Phi_{\ell_1\ell_2}^{(N)}\beta, (\mu_{\bx}^{(d_n)})^{\otimes(N-1)}\ra - \la \beta, (\mu_{\bx}^{(d_n)})^{\otimes N}\ra \Big)+\la \mathcal{Q}^{(N)}\beta, (\mu_{\bx}^{(d_n)})^{\otimes N}\ra \\
    \end{aligned}
    \]
    for $\bx=(x_1, x_2, \dots, x_{d_n}) \in \Delta_{d_n-1}$, 
    where we set
    \[
    \mu_{\bx}^{(d_n)}=\sum_{i=1}^{d_n}x_i\delta_{z_i^{(d_n)}}
    \in \mathcal{P}(E^{(d_n)}) \subset \mathcal{P}(E). 
    \]
    Moreover, we can easily see that 
    \begin{align}
    \del_i(P_{d_n}\varphi)(\bx)
    &=\sum_{\ell_1=1}^N \left(
    \sum_{i_1, \dots, i_N=1}^{d_n} 
    \prod_{m \neq \ell_1}x_{i_m}
    \delta_{ii_{\ell_1}}\beta(z_{i_1}^{(d_n)},  \dots, z_{i_{N}}^{(d_n)})\right), 
    \label{Eq:first-deriv}\\
    \del_{ij}(P_{d_n}\varphi)(\bx)
    &=\sum_{\ell_1, \ell_2=1}^N \left(
    \sum_{i_1, \dots, i_N=1}^{d_n}
    \prod_{m \neq \ell_1, \ell_2}x_{i_m} 
    \delta_{ii_{\ell_1}}\delta_{ji_{\ell_2}}
    \beta(z_{i_1}^{(d_n)},  \dots, z_{i_{N}}^{(d_n)})\right), 
    \label{Eq:second-deriv}\\
    \del_{ijk}(P_{d_n}\varphi)(\bx)
    &=\sum_{\ell_1, \ell_2, \ell_3=1}^N \left(
    \sum_{i_1, \dots, i_N=1}^{d_n}
    \prod_{m \neq \ell_1, \ell_2, \ell_3}x_{i_m} 
    \delta_{ii_{\ell_1}}\delta_{ji_{\ell_2}}\delta_{ki_{\ell_3}}
    \beta(z_{i_1}^{(d_n)},  \dots, z_{i_{N}}^{(d_n)})\right).
    \label{Eq:third-deriv}
    \end{align}
    Here, we apply Taylor's formula to get
    \begin{align}
     & n(B_{d_n, n}^{(\mathbf{q}_n)}-I)P_{d_n}\varphi(\bx)  \nn \\
     &=n\sum_{i=1}^{d_n} 
     \del_{i}(P_{d_n}\varphi)(\bx)
     \mathbb{E}\big[(G_n^{(\mathbf{q}_n)})^i(\bx) - x_i\big]
     \nn \\
     &\hspace{1cm}+\frac{n}{2}\sum_{i, j=1}^{d_n}
     \del_{ij}(P_{d_n}\varphi)(\bx) 
     \mathbb{E}\big[((G_n^{(\mathbf{q}_n)})^i(\bx) - x_i)
     ((G_n^{(\mathbf{q}_n)})^j(\bx) - x_j)\big] \nn \\
     &\hspace{1cm}+\frac{n}{2}\sum_{i, j, \ell=1}^{d_n}
     \int_0^1 (1-t)^2
     \del_{ij\ell}(P_{d_n}\varphi)(\bx+t(G_n^{(\mathbf{q}_n)}(\bx)-\bx)) \, \dd t  \nn \\
     &\hspace{1.5cm}
     \times 
     \mathbb{E}\big[((G_n^{(\mathbf{q}_n)})^i(\bx) - x_i)
     ((G_n^{(\mathbf{q}_n)})^j(\bx) - x_j)
     ((G_n^{(\mathbf{q}_n)})^k(\bx) - x_k)\big]\nn\\
     &=: I_n^{(1)}(\bx)+I_n^{(2)}(\bx)
     +I_n^{(3)}(\bx), \qquad 
     \bx \in \Delta_{d_n-1}.
     \label{Eq:Taylor-99}
    \end{align}
We first consider the terms $I_n^{(1)}(\bx)$ and $I_n^{(2)}(\bx)$. 
In view of \eqref{Eq:first-deriv} and \eqref{Bernstein-1st-moment-general}, 
the term $I_n^{(1)}(\bx)$ is going to be
\begin{align}
    I_n^{(1)}(\bx) &=
    n\sum_{i=1}^{d_n} 
    \sum_{\ell_1=1}^N \left(
    \sum_{i_1, \dots, i_N=1}^{d_n} 
    \prod_{m \neq \ell_1}x_{i_m}
    \delta_{ii_{\ell_1}}\beta(z_{i_1}^{(d_n)}, z_{i_2}^{(d_n)}, \dots, z_{i_{N}}^{(d_n)})\right)
    h_n^i(\bx) \nn \\
    &=n\sum_{\ell_1=1}^N
    \sum_{\substack{1 \le i_m \le d_n \\ m \neq \ell_1}}
    \prod_{m \neq \ell_1}x_{i_m}
    \sum_{i=1}^{d_n}
    \beta(z_{i_1}^{(d_n)}, \dots, \overbrace{z_i^{(d_n)}}^{\text{$i_{\ell_1}$-th}}, \dots, z_{i_N}^{(d_n)})\left(\sum_{j=1}^{d_n}q_{ji}^{(n)}x_j\right) \nn \\
    &= n\sum_{\ell_1=1}^N
    \sum_{\substack{1 \le i_m \le d_n \\ m \neq \ell_1}}
    \prod_{m \neq \ell_1}x_{i_m}
    \sum_{i=1}^{d_n}x_i
    \left(\sum_{j=1}^{d_n}q_{ij}^{(n)}
    \beta(z_{i_1}^{(d_n)}, \dots, \overbrace{z_i^{(d_n)}}^{\text{$i_{\ell_1}$-th}}, \dots, z_{i_N}^{(d_n)})\right) \nn \\
    &=\la n\mathcal{Q}_n^{(N)}\beta, (\mu_{\bx}^{(d_n)})^{\otimes N}\ra, \qquad \bx \in \Delta_{d_n-1}.
    \label{Eq:infty-I_1}
\end{align}
Furthermore, it follows from \eqref{Eq:second-deriv}
and \eqref{Bernstein-2nd-moment-general} that 
\begin{align}
    I_n^{(2)}(\bx) &=
    \frac{1}{2}\sum_{i, j=1}^{d_n} 
    \sum_{\ell_1, \ell_2=1}^N \left(
    \sum_{i_1, \dots, i_N=1}^{d_n}
    \prod_{m \neq \ell_1, \ell_2}x_{i_m} 
    \delta_{ii_{\ell_1}}\delta_{ji_{\ell_2}}
    \beta(z_{i_1}^{(d_n)}, z_{i_2}^{(d_n)}, \dots, z_{i_{N}}^{(d_n)})\right)
     \nn \\
    &\hspace{1cm}\times \left(
    x_i(\delta_{ij} - x_j) 
        +h_n^i(\bx)(\delta_{ij}-x_j)
        - x_ih_n^j(\bx)
        +(n+1)
        h_n^i(\bx) h_n^j(\bx)\right) \nn \\
    &=\frac{1}{2}\sum_{\ell_1, \ell_2=1}^N
    \sum_{\substack{1 \le i_m \le d_n \\ m \neq \ell_1, \ell_2}}
    \prod_{m \neq \ell_1, \ell_2}x_{i_m} 
    \sum_{i, j=1}^{d_n} x_i(\delta_{ij}-x_j)
    \beta(z_{i_1}^{(d_n)}, \dots, 
    \overbrace{z_i^{(d_n)}}^{\text{$i_{\ell_1}$-th}}, \dots, 
    \overbrace{z_j^{(d_n)}}^{\text{$i_{\ell_2}$-th}}, \dots, z_{i_{N}}^{(d_n)})
     \nn \\
    &\hspace{1cm}+\mathcal{G}_n(\beta; \bx) \nn \\
    &= \sum_{1 \le \ell_1<\ell_2 \le N}\Bigg\{
    \sum_{\substack{1 \le i_m \le d_n \\ m \neq \ell_1, \ell_2}}
    \prod_{m \neq \ell_1, \ell_2}x_{i_m} 
    \sum_{i=1}^{d_n} x_i
    \beta(z_{i_1}^{(d_n)}, \dots, 
    \overbrace{z_i^{(d_n)}}^{\text{$i_{\ell_1}$-th}}, \dots, 
    \overbrace{z_i^{(d_n)}}^{\text{$i_{\ell_2}$-th}}, \dots, z_{i_{N}}^{(d_n)})
     \nn \\
    &\hspace{1cm}-\sum_{\substack{1 \le i_m \le d_n \\ m \neq \ell_1, \ell_2}}
    \prod_{m \neq \ell_1, \ell_2}x_{i_m} 
    \sum_{i, j=1}^{d_n} x_ix_j
    \beta(z_{i_1}^{(d_n)}, \dots, 
    \overbrace{z_i^{(d_n)}}^{\text{$i_{\ell_1}$-th}}, \dots, 
    \overbrace{z_j^{(d_n)}}^{\text{$i_{\ell_2}$-th}}, \dots, z_{i_{N}}^{(d_n)})\Bigg\} \nn \\
    &\hspace{1cm}+\mathcal{G}_n(\beta; \bx) \nn \\
    &=\sum_{1 \le \ell_1 < \ell_2 \le N}\Big(
    \la \Phi_{\ell_1\ell_2}^{(N)}\beta, (\mu_{\bx}^{(d_n)})^{\otimes(N-1)}\ra - \la \beta, (\mu_{\bx}^{(d_n)})^{\otimes N}\ra \Big)
    +\mathcal{G}_n(\beta; \bx)
    \label{Eq:infty-I_2}
\end{align}
for each $\bx \in \Delta_{d_n-1}$, where we put 
\[
\begin{aligned}
\mathcal{G}_n(\beta; \bx)
&=\frac{1}{2}\sum_{\ell_1, \ell_2=1}^N
    \sum_{\substack{1 \le i_m \le d_n \\ m \neq \ell_1, \ell_2}}
    \prod_{m \neq \ell_1, \ell_2}x_{i_m} 
    \sum_{i, j=1}^{d_n} 
    \beta(z_{i_1}^{(d_n)}, \dots, 
    \overbrace{z_i^{(d_n)}}^{\text{$i_{\ell_1}$-th}}, \dots, 
    \overbrace{z_j^{(d_n)}}^{\text{$i_{\ell_2}$-th}}, \dots, z_{i_{N}}^{(d_n)})
     \nn \\
    &\hspace{1cm}\times \left(
        h_n^i(\bx)(\delta_{ij}-x_j)
        - x_ih_n^j(\bx)
        +(n+1)
        h_n^i(\bx) h_n^j(\bx)\right),
        \qquad \bx \in \Delta_{d_n-1}. 
\end{aligned}
\]
Here, the norm of $\mathcal{G}_n(\beta; \bx)$ is estimated as 
\begin{align}
    \|\mathcal{G}_n(\beta; \cdot)\|_\infty
    &\le \frac{1}{2}\|\beta\|_\infty N^2d_n^3
    (2Ca_n+Ca_n+Cna_n^2) \nn \\
    &\le C(d_n^3a_n+nd_n^3a_n^2) \nn. 
\end{align}
Thanks to {\bf (D)} and {\bf (Q2)}, 
we see that 
$d_n^3a_n=o(n^{-5/16})$ and $nd_n^3a_n^2=o(1)$ as $n \to \infty$,
which leads to 
\begin{equation}\label{Eq:g_n to 0}
    \lim_{n \to \infty}\|\mathcal{G}_n(\beta; \cdot)\|_\infty=0. 
\end{equation}
On the other hand, 
the term $I_n^{(3)}(\bx)$ is estimated as 
\begin{align}
    &|I_n^{(3)}(\bx)| \nn \\
    &\le
     \frac{n}{2}\sum_{i, j, \ell=1}^{d_n}
     \int_0^1 (1-t)^2
     |\del_{ij\ell}(P_{d_n}\varphi)(\bx+t(G_n^{(\mathbf{q}_n)}(\bx)-\bx))| \, \dd t  \nn \\
     &\hspace{1.5cm}
     \times 
     \mathbb{E}\big[|(G_n^{(\mathbf{q}_n)})^i(\bx) - x_i||
     (G_n^{(\mathbf{q}_n)})^j(\bx) - x_j||
     (G_n^{(\mathbf{q}_n)})^k(\bx) - x_k|\big]\nn\\
     & \le \frac{n}{2}\|\beta\|_\infty\sum_{i, j, \ell=1}^{d_n}
     \int_0^1 (1-t)^2 \sum_{\ell_1, \ell_2, \ell_3=1}^N \left(
    \sum_{\substack{1 \le i_m \le d_n \\ m \neq \ell_1, \ell_2, \ell_3}}
    \prod_{m \neq \ell_1, \ell_2, \ell_3}(x_{i_m}+t((G_n^{(\mathbf{q}_n)})^{i_m}-x_{i_m}) \right) \, \dd t \nn\\
    &\hspace{1cm}\times 
    \E\Big[((G_n^{(\mathbf{q}_n)})^i(\bx) - x_i)^2
     ((G_n^{(\mathbf{q}_n)})^j(\bx) - x_j)^2\Big]^{1/2}
    \E\Big[((G_n^{(\mathbf{q}_n)})^k(\bx) - x_k)^2\Big]^{1/2} \nn \\
    &\le \frac{n}{6}\|\beta\|_\infty N^3 d_n\sum_{i, j, \ell=1}^{d_n}
    \E\Big[((G_n^{(\mathbf{q}_n)})^i(\bx) - x_i)^4\Big]^{1/4}
    \E\Big[((G_n^{(\mathbf{q}_n)})^j(\bx) - x_j)^4\Big]^{1/4} \nn \\
    &\hspace{1cm} \times \E\Big[((G_n^{(\mathbf{q}_n)})^k(\bx) - x_k)^2\Big]^{1/2},
    \qquad \bx \in \Delta_{d_n-1},
    \label{Eq:Taylor-99-3}
\end{align}
where we applied \eqref{Eq:third-deriv} 
and the Schwarz inequality. 
Since it holds that 
\[
\begin{aligned}
&\E\Big[((G_n^{(\mathbf{q}_n)})^i(\bx) - x_i)^{2\beta}\Big]\\
&\le 2^{2\beta-1}\left\{
\E\Big[((G_n^{(\mathbf{q}_n)})^i(\bx) - x_i^{(\mathbf{q}_n)})^{2\beta}\Big]+(x_i^{(\mathbf{q}_n)}-x_i)^{2\beta}\right\} \\
&\le 2^{2\beta-1}\left\{\frac{1}{n^{2\beta}} \times Cn^{\beta}+\left(\frac{C}{n}\right)^{2\beta}\right\}
\le \frac{C}{n^{\beta}}, \qquad \beta \in \N, \, \bx \in \Delta_{d_n-1},
\end{aligned}
\]
due to Theorem \ref{Thm:binomial-moment}, \eqref{Eq:Taylor-99-3} is going to be
\begin{equation}\label{Eq:I_n^3 to 0}
|I_n^{(3)}(\bx)| \le \frac{n}{6}\|\beta\|_\infty N^3 d_n^4 \times 
\left(\frac{C}{n^2}\right)^{1/2} \times \left(\frac{C}{n}\right)^{1/2}
\le C \left(\frac{d_n}{n^{1/8}}\right)^4
\to 0
\end{equation}
as $n \to \infty$ for every $\bx \in \Delta_{d_n-1}$, thanks to $d_n=o(n^{1/8})$. 
Therefore, \eqref{Eq:Taylor-99} together with {\bf (Q3)}, \eqref{Eq:infty-I_1},
\eqref{Eq:infty-I_2}, 
\eqref{Eq:g_n to 0} and \eqref{Eq:I_n^3 to 0} implies that
\begin{align}\label{Eq:convergence-final}
    &\|n(B_{d_n, n}^{(\mathbf{q}_n)}-I)P_{d_n}\varphi
     - P_{d_n}\mathfrak{A}\varphi \|_\infty \nn \\
     &\le \|n\mathcal{Q}_n^{(N)}\beta - \mathcal{Q}^{(N)}\beta\|_\infty
     +\|\mathcal{G}_n(\beta; \cdot)\|_\infty + \|I_n^{(3)}(\cdot)\|_\infty \to 0
     \nn
\end{align}
as $n \to \infty$, which is the very desired convergence. 

Furthermore, under Assumptions {\bf (A5)} and {\bf (A6)}, 
we have 
\[
\begin{aligned}
    |I_n^{(1)}(\bx)| \le C\tau_n(\beta), \qquad 
    |I_n^{(2)}(\bx)| \le Cn^{-5/16-2\ve}, \qquad 
    |I_n^{(3)}(\bx)| \le Cn^{-\ve}
\end{aligned}
\]
for $n \in \N$ and $\bx \in \Delta_{d_n-1}$, 
which readily implies \eqref{Eq:Voronovskaya-infinite-rate}.
\end{proof}

\subsection{Proof of Theorem \ref{Thm:Bernstein-infinity-1}}

We now give the proof of Theorem \ref{Thm:Bernstein-infinity-1}. 

\begin{proof}[Proof of Theorem {\rm \ref{Thm:Bernstein-infinity-1}}]
    Since $\Dom(\mathfrak{A})$ is densely defined and 
    $(\lambda - \mathfrak{A})(\mathcal{D})$ is dense in
    $C(\mathcal{P}(E))$ by noting Lemma \ref{Lem:core-infinity}, Trotter's approximation theorem (cf.~Proposition \ref{Prop:Trotter}) implies that
    \begin{equation}
            \lim_{n \to \infty}
            \left\|(B_{d_n,n}^{(\mathbf{q}_n)})^{\lfloor nt \rfloor}P_{d_n}\varphi 
            - P_{d_n}{\sf T}_t^{(\infty)}\varphi\right\|_\infty=0
            , \qquad \varphi \in C(\mathcal{P}(E)), \, t \ge 0. \nn
            \label{Eq:Bernstein-convergence-infinity-2}
    \end{equation}
    Now let $(\widetilde{\mathsf{T}}_t^{(\infty)})_{t \ge 0}$
    be the $\mathfrak{A}$-diffusion semigroup on $C(\mathcal{P}(E))$ defined by 
    $$
    \widetilde{\mathsf{T}}_t^{(\infty)}\varphi(\mu)
    =\E\big[\varphi\big({\sf X}_t^{(\infty)}(\mu)\big)\big], 
        \qquad \varphi \in C(\mathcal{P}(E)), \, 
        \mu \in \mathcal{P}(E), \, t \ge 0.
    $$
    The rest is to show that the diffusion semigroup $(\widetilde{\mathsf{T}}_t^{(\infty)})_{t \ge 0}$ on $C(\mathcal{P}(E))$
    coincides with the Feller semigroup $(\mathsf{T}_t^{(\infty)})_{t \ge 0}$ on $C(\mathcal{P}(E))$. 
    The following argument is based on a technique 
    demonstrated in \cite[Section 3]{EG}.

    For $N \in \N$ and $\beta_1, \beta_2, \dots, \beta_N \in C(E)$,  define a function 
    $\varphi_{\beta_1, \beta_2, \dots, \beta_N} \in \mathcal{D} \subset C(\mathcal{P}(E))$ by  
    \[
    \varphi_{\beta_1, \beta_2, \dots, \beta_N}(\mu)=
    \la \beta_1, \mu\ra \la \beta_2, \mu \ra \cdots 
    \la \beta_N, \mu\ra. 
    \]
   Then, we easily have 
    \begin{align}
        \mathsf{T}_t^{(\infty)}\varphi_{\beta_1, \beta_2, \dots, \beta_N}(\mu)
        &= \varphi_{\beta_1, \beta_2, \dots, \beta_N}(\mu)
        + \int_0^t \mathsf{T}_s^{(\infty)}\mathfrak{A}
        \varphi_{\beta_1, \beta_2, \dots, \beta_N}(\mu) \, \dd s, 
        \label{Eq:semigroup-comparison-1}\\
        \widetilde{\mathsf{T}}_t^{(\infty)}\varphi_{\beta_1, \beta_2, \dots, \beta_N}(\mu)
        &= \varphi_{\beta_1, \beta_2, \dots, \beta_N}(\mu)
        + \int_0^t \widetilde{\mathsf{T}}_s^{(\infty)}\mathfrak{A}
        \varphi_{\beta_1, \beta_2, \dots, \beta_N}(\mu) \, \dd s
        \label{Eq:semigroup-comparison-2}
    \end{align}
for $t \ge 0$ and $\mu \in \mathcal{P}(E)$. 
We now define a function $K_N : [0, \infty) \to [0, \infty)$ by 
\[
K_N(t):=\sup_{\substack{\beta_1, \beta_2, \dots, \beta_N \in C(E) \\ 
\|\beta_1\|_\infty \vee \cdots \vee  \|\beta_N\|_\infty \le 1}}\sup_{\mu \in \mathcal{P}(E)}
|\mathsf{T}_t^{(\infty)}\varphi_{\beta_1, \beta_2, \dots, \beta_N}(\mu)
-\widetilde{\mathsf{T}}_t^{(\infty)}\varphi_{\beta_1, \beta_2, \dots, \beta_N}(\mu)|.
\]
By combining the identity 
\[
\begin{aligned}
    \mathfrak{A}\varphi_{\beta_1, \beta_2, \dots, \beta_N}(\mu)
    &= \sum_{1 \le i < j \le N}\Big(
    \la \beta_i\beta_j, \mu\ra - \la \beta_i, \mu \ra \la \beta_j, \mu\ra \Big) 
    \prod_{\ell \neq i, j}\la \beta_\ell, \mu\ra \\ 
    &\hspace{1cm}+\sum_{i=1}^N \la \mathcal{Q}\beta_i, \mu\ra 
    \prod_{\ell\neq i}\la \beta_\ell, \mu\ra, \qquad \mu \in \mathcal{P}(E),
\end{aligned}
\]
with \eqref{Eq:semigroup-comparison-1} and \eqref{Eq:semigroup-comparison-2}, one has
\[
K_N(t) \le \int_0^t \left(2 \times \frac{N(N-1)}{2}K_N(s)+NK_N(s)\right) \, \dd s=N^2\int_0^t K_N(s) \, \dd s, \qquad t \ge 0.
\]
Thus, the Gronwall inequality immediately leads to $K_N(t) \equiv 0$, which 
concludes that $\mathsf{T}_t^{(\infty)}=\widetilde{\mathsf{T}}_t^{(\infty)}$ for $t \ge 0$ and \eqref{Eq:Bernstein-convergence-infinity} has been proved. 

Finally, we assume {\bf (A5)} and {\bf (A6)}. For $n \in \N$ and 
$\ve \in (0, 1/8)$,  
we have 
    \[
    \begin{aligned}
    \|n(B_{d_n, n}^{(\mathbf{q}_n)}-I)P_{d_n}\varphi\|_{\infty} 
    &\le C(\tau_n(\beta) \vee n^{-\ve})+\|\mathfrak{A}\varphi\|_\infty, \\
    \|n(B_{d_n, n}^{(\mathbf{q}_n)}-I)P_{d_n}\varphi
     - P_{d_n}\mathfrak{A}\varphi \|_\infty 
    &\le C(\tau_n(\beta) \vee n^{-\ve}) \to 0 \quad \text{as $n \to \infty$}. 
    \end{aligned}
    \]
    Therefore, Proposition 
    \ref{Prop:Trotter-rate} yields the
    rate of convergence \eqref{Eq:Bernstein-convergence-infinity-rate}.
\end{proof}

\section{{\bf Conclusions and further comments}}
\label{Sect:Conclusion}


In the present paper, we have discussed limit theorems 
for iterates of the multidimensional Bernstein operator
$B_{d, n}^{(\mathbf{q}_n)}$ associated with a certain 
set of real numbers $\mathbf{q}_n$ and have obtained 
some explicit relations between the limiting objectives 
and the multidimensional 
diffusion processes arising in the study of population genetics. 
Furthermore, we also capture a measure-valued diffusion process
in population genetics as the limit of the iterate of 
the multidimensional Bernstein operator associated with $\mathbf{q}_n$
as not only the number of iterate but also the dimension 
tend to infinity. 
Some of our results in the present paper seem to be whole new. 
Hence, we believe that our results does make significant contributions to 
the development of the study of many areas of mathematics such as 
probability theory, approximation theory and functional analysis
as well as population genetics. 
On the other hand, we also find out several possible directions 
to do next, all of which are to be worth further investigations.

In the present paper, we do not discuss the 
$d$-dimensional Wright--Fisher diffusion with both mutation and {\it selection}. 
Such a diffusion has an infinitesimal generator of the form
\[
\begin{aligned}
\mathcal{A}_df(\bx)
&=\frac{1}{2}\sum_{i,j=1}^d x_i(\delta_{ij}-x_j)
\frac{\del^2 f}{\del x_i \del x_j} \\
&\hspace{1cm}
+\sum_{i=1}^d \left(
\sum_{j=1}^d q_{ji}x_j + x_i
\left(\sum_{j=1}^d s_{ij}x_j - 
\sum_{j, k=1}^d s_{jk}x_jx_k\right)\right)\frac{\del f}{\del x_i}
\end{aligned}
\]
acting on $C^2(\Delta_{d-1})$, where 
$s_{ij}=s_{ji} \in \R$, $i, j=1, 2, \dots, d$, 
is the fitness parameter of selections.
Since the drift coefficient is no longer linear in each $x_i$, 
it might be slightly difficult to capture the diffusion 
in terms of the limit of the iterate of some 
extended Bernstein operator, while it is definitely an ideal 
goal. 
If done, the result would make an extensive 
contribution to the study of population genetics 
and related mathematics.

Another positive linear operator of interest is
the {\it Sz\'asz--Mirakyan operator} (cf.~\cite{Szasz, Mirakyan}), which is a kind of extension of the 
Bernstein operator. 
For $n \in \N$, define 
\[
M_nf(x):=\sum_{k=0}^\infty e^{-nx}\frac{(nx)^k}{k!}f\left(\frac{k}{n}\right),
\qquad f \in C([0, \infty)), \, x \in [0, \infty). 
\]
In particular, Sz\'asz proved that the positive linear operator $M_n$
uniformly approximates every continuous function in $C_\infty([0, \infty))$, 
the Banach space of all continuous functions $f : [0, \infty) \to \R$
vanishing at infinity. Motivated by the convergence result of \cite{KYZ18}, 
one of the author of the present paper have discussed in \cite{ANS} 
limit theorems for the iterate of $M_n$. 
By noting that $M_nf$ can be expressed as the expectation with respect to 
the Poisson distribution, 
it has been shown that the iterate $M_n^{\lfloor nt \rfloor}$
uniformly converges to the one-dimensional diffusion semigroup on a 
weighted function space which
corresponds to the stochastic differential equation 
\[
\dd \mathsf{Y}_t(x)=\sqrt{\mathsf{Y}_t(x)} \, \dd W_t, 
\qquad \mathsf{Y}_0(x)=x \in [0, \infty),
\]
where $(W_t)_{t \ge 0}$ is a one-dimensional 
standard Brownian motion (cf.~\cite[Theorem 1.5]{ANS}). 
Then, we wonder if its multidimensional extension can also be established or not, 
which should be an interesting problem.
However, there are some technical difficulties 
in establishing some limit theorems for the iterates of a
$d$-dimensional extension of the Sz\'asz--Mirakyan operator 
on $C([0, \infty)^d)$, 
even though it can be properly defined. 
One of them is how we should set
the boundary conditions for the 
infinitesimal generator corresponding to the limiting diffusion, since the diffusion coefficient is degenerate 
in the whole boundary. 
An effective way to overcome the difficulty 
may be to create a positive drift in the limiting diffusion 
so as to avoid hitting the boundary. 
Hence, we need to extend the definition of the 
multidimensional Sz\'asz--Mirakyan operator 
similarly to Definition \ref{Def:modified Bernstein}
in the Bernstein case.
Furthermore, its infinite-dimensional extension 
would also be a quite intriguing problem to do. 
In particular, we are interested in what kinds of 
measure-valued diffusion processes can be captured 
in the Sz\'asz--Mirakyan case.

On the other hand, yet another extension of 
the Bernstein operator might be realized by replacing 
the binomial distribution in the definition by 
a one-dimensional probability distribution which differs from both binomial and Poisson ones. 
Some of such examples have already discussed in 
e.g. \cite{KR67, Tate, Altomare}, mainly in the context of 
approximation theory. 
In order to try to generalize this idea including these 
existing examples, we may introduce
the {\it generalized Bernstein operator}
\[
\mathfrak{B}_nf(x):=
\int_E f\left(\frac{t}{n}\right) \, \mu_x^{*n}(\dd t), 
\qquad x \in E,
\]
acting on 
some Banach space consisting of continuous functions 
on a locally compact metric space $E \subset \R$, 
where $\mu_x^{*n}$ stands for the $n$-fold convolution of
a certain probability measure supported on $E$
satisfying 
$
\int_E t \, \mu_x(\dd t)=x.
$
Under this generalization, we expect to
obtain various kinds of limit theorems and 
diffusion processes taking values in $E$, 
which should also be fascinating in a probabilistic point of view.
However, there still be a lot of essential difficulties to 
establish the general theory.
So we postpone these problems in the forthcoming paper.

\section*{{\bf Acknowledgement}}
The authors would like to thank
Professor Hiroshi Kawabi for offering extremely helpful suggestions to infinite-dimensional generalizations,
which makes the quality of the present paper much better. 
We also thank Professor Jir\^o Akahori, who was the advisor of the first-named author,  
for valuable comments and constant encouragement.
The views or opinions expressed in the present paper 
do not necessarily reflect those of Mizuho-DL Financial Technology Co., Ltd. 
The second-named author is supported 
by JSPS KAKENHI Grant Number 19K23410 and 23K12986.

\section*{{\bf Declaration}}

\noindent
{\bf Conflict of interest}\quad The authors declare that they have no conflict of interest.

\begin{appendix}
\label{app}

\color{black}
\section{{\bf Trotter's approximation theorem and 
its rate of convergence}}
\label{App:Trotter}

{\it Trotter's approximation theorem} gives a sufficient condition for 
the iterate of a bounded linear operator 
acting on a Banach space to converge to a $ C_0 $-semigroup. 
We give the statement of Trotter's approximation theorem 
when the linear operator enjoys the contraction property. 
Let $(\mathcal{U}_n, \|\cdot\|_{\mathcal{U}_n}), \, n \in \mathbb{N}$, 
and $(\mathcal{U}, \|\cdot\|_{\mathcal{U}})$ be 
 Banach spaces. We denote by $P_n : \mathcal{U} \to \mathcal{U}_n, \, n \in \mathbb{N},$ 
 a bounded linear operator
 with $\|P_n\| \le 1$ for $n \in \mathbb{N}$. 
We say that the sequence of pairs $\{(\mathcal{U}_n, P_n)\}_{n=1}^\infty$ approximates 
the Banach space $\mathcal{U}$ if  
$\|P_n f\|_{\mathcal{U}_n} \to \|f\|_{\mathcal{U}}$ as $n \to \infty$
 for $f \in \mathcal{U}$, 
which means that each $P_n, \, n \in \mathbb{N},$ 
is regarded as an isomorphism between $\mathcal{U}_n$ and $\mathcal{U}$ in passing 
to the limit. 
Let $f_n \in \mathcal{U}_n$ and $f \in \mathcal{U}$. 
We also say that 
$f=\lim_{n \to \infty}f_n$
if 
$\|f_n-P_nf\|_{\mathcal{U}_n} \to 0$ as $n \to \infty$.  
Then, we define 
the limit $\mathcal{A}$ of a sequence of 
linear operators $\mathcal{A}_n$ by putting 
$$
\begin{aligned}
\mathcal{A}f &:= \lim_{n \to \infty}\mathcal{A}_nP_n f, \qquad f \in \mathcal{U}, \\
\mathrm{Dom}(\mathcal{A})&:= \left\{f \in \mathcal{U} \, \Big| \, 
\lim_{n \to \infty}\mathcal{A}_nP_n f \text{ exists}\right\}. 
\end{aligned}
$$
Then, we have the following.

\begin{pr}[cf.~{\cite[Theorem 5.1]{Trotter}} and {\cite[Theorem 2.13]{Kurtz}}]
\label{Prop:Trotter}
For $n \in \N$, let $ T_n $ be a bounded linear operator on $ \mathcal{U}_n $ satisfying 
$ \|T_n\| \le 1 $. 
Put $ \mathcal{A}_n:=n(T_n-I), \, n \in \N. $
We define a linear operator $ \mathcal{A} $ 
by the closure of the limit of $ \mathcal{A}_n $.  
If the domain $ \Dom(\mathcal{A}) $ is dense in $ \mathcal{U} $ and 
the range of $ \lambda - \mathcal{A} $ is dense 
in $ \mathcal{U} $ for some $ \lambda>0 $, 
then  there exists a 
$ C_0 $-semigroup $ (\mathsf{T}_t)_{t \ge 0} $
acting on $ \mathcal{U} $ 
satisfying $\|\mathsf{T}_t\| \le 1, \, t \ge 0$, and 
    \begin{equation}\label{limit-Trotter}
    \lim_{n \to \infty}
    \|T_n^{\lfloor nt \rfloor}P_nf - P_n\mathsf{T}_tf\|_{\mathcal{U}_n}=0,
    \qquad t \ge 0, \, f \in \mathcal{U}. 
    \end{equation}
\end{pr}

On the other hand, 
Proposition \ref{Prop:Trotter} does not imply 
a quantitative estimates of \eqref{limit-Trotter}. 
Campiti and Tacelli established in \cite{CT} 
a refinement of Proposition \ref{Prop:Trotter} by giving the rate of convergence 
of \eqref{limit-Trotter}, under the condition $\mathcal{U}_n \equiv \mathcal{U}$. 
Afterwards, one of the authors of the present paper  also established in \cite{Namba} for the convergence rate corresponding to Trotter's approximation theorem 
in more general settings, whose statement is as follows:

\begin{pr}[cf.~{\cite[Theorem 1]{Namba}}]
\label{Prop:Trotter-rate}
Suppose that  $ \{T_n\}_{n=1}^\infty $, 
$ \{\mathcal{A}_n\}_{n=1}^\infty $, $\mathcal{A}$ and $ (\mathsf{T}_t)_{t \ge 0} $
are as in Proposition {\rm \ref{Prop:Trotter}}.
Let $ \mathcal{D} $ be a dense subspace of $ \mathcal{U} $. 
We assume that
    \[
    \|\mathcal{A}_nP_n f\|_{\mathcal{U}_n} \le \varphi_n(f), \qquad 
    \|\mathcal{A}_nP_n f - P_n\mathcal{A}f\|_{\mathcal{U}_n} \le \psi_n(f), 
    \qquad f \in \mathcal{D},
    \]
where $ \varphi_n, \psi_n : \mathcal{D} \to [0, \infty) $
are semi-norms with $ \psi_n(f) \to 0 $ as $ n \to \infty $
for $ f \in \mathcal{D} $.
Then, for  $n \in \N$, $ t \ge 0 $ and 
$ f \in \{g \in \mathcal{D} \, 
| \, \mathsf{T}_t g \in \mathcal{D}, \, t \ge 0\} $, we have 
    \begin{equation}\label{limit-Trotter2}
    \|T_n^{\lfloor nt \rfloor}P_n f - P_n\mathsf{T}_tf\|_{\mathcal{U}_n}
    =\sqrt{\frac{t}{n}}\varphi_n(f)+\frac{1}{n}\varphi_n(f)
    +\int_0^t \psi_n(\mathsf{T}_sf) \, \dd s.
    \end{equation}
\end{pr}

\section{{\bf Central moments for the
binomial distribution}}
\label{App:moments}

Although the binomial distribution is a 
representative of typical discrete probability distributions, 
there have been only a few papers
in which their exact raw and central moment formulae 
are exhibited. 
In fact, some recursive formulae for such moments 
have been obtained in some papers. 
However, it is quite difficult to
write down their general forms explicitly 
since some combinatorial numbers 
such as the Stirling number of the second kind
appear therein. 
This appendix presents some estimates of
central moments of even order 
for the binomial distribution. 
This estimate plays  significant roles 
in the proofs of Theorems \ref{Thm:Bernstein-2}
and \ref{Thm:Bernstein-general-2}. 

Suppose that a random variable $S$  follows 
the binomial distribution with parameters $n \in \N$
and $x \in [0, 1]$, 
that is, 
    \[
    \begin{aligned}
    \mathbb{P}(S=k)&=\binom{n}{k}p^k(1-p)^{n-k},
    \qquad k=0, 1, 2, \dots, n. \\
    \end{aligned}
    \]
Then, it is easy to see that $\E[S]=nx$. 
The things we want to know are some upper estimates of 
the central moment $\E[(S-nx)^\gamma]$
for $\gamma \in \N$. 
Skorski \cite{Skorski} obtained the following 
behavior of the central moment 
of $S$ of even order. 

\begin{pr}[cf.~{\cite[Theorem 3]{Skorski}}]
\label{prop:Skorski-moment}
Let $\gamma=2\beta \ge 4$ be an even integer
and put $\sigma^2:=x(1-x)$. 
Then, we have 
    \[
    \E[(S-nx)^{2\beta}]^{1/2\beta}
    =\Theta(1) \cdot \max_{k=2, 3, \dots, \beta} 
    k^{1-k/2\beta} \cdot (n\sigma^2)^{k/2\beta}, 
    \]
where $f(N)=\Theta(g(N))$ means that 
\[
c_1g(N) \le f(N) \le c_2g(N), \qquad 
N \ge N_0,
\]
holds for 
some $c_1, c_2>0$ and some $N_0 \in \N$. 
\end{pr}

By making use of this proposition, 
we establish the following upper bound. 

\begin{tm}
\label{Thm:binomial-moment}
Let $\gamma=2\beta$ be an even positive integer. 
Then, we have 
    \begin{equation*}\label{C-1}
    \E[(S-nx)^{2\beta}]
    \le Cn^\beta, \qquad n \in\N,
    \end{equation*}
    for some $C>0$ independent of $n$. 
\end{tm}

\begin{proof}
If $\gamma=2\beta=2$, the central moment is 
nothing but the variance of $S$ and 
it is obvious that $\mathrm{Var}(S)=nx(1-x) \le n/4$. 
Suppose that $\gamma=2\beta \ge 4$. 
Then, there is a sufficiently 
large $C>0$ depending only on $\beta$ such that
    \[
    \max_{k=2, 3, \dots, \beta}
    k^{1-k/2\beta}(\sigma^2)^{k/2\beta}\le C,
    \]
since it holds that $\sigma^2=xy=x(1-x) \le 1/4$ uniformly in $x \in [0, 1]$.
Therefore, we obtain
    \[
    \E[(S-nx)^{2\beta}]^{1/2\beta}
    \le C \cdot \max_{k=2, 3, \dots, \beta}n^{k/2\beta}=Cn^{1/2}
    \]
as soon as we choose $C$ as being sufficiently large. 
\end{proof}

\end{appendix}


\end{document}